\documentclass[a4paper]{amsart}
\usepackage{tikz-cd}
\usetikzlibrary{decorations.markings}
\makeatletter
\tikzcdset{
  open/.code     = {\tikzcdset{hook, circled};},
  closed/.code   = {\tikzcdset{hook, slashed};},
  open'/.code    = {\tikzcdset{hook', circled};},
  closed'/.code  = {\tikzcdset{hook', slashed};},
  circled/.code  = {\tikzcdset{markwith = {\draw (0,0) circle (.375ex);}};},
  slashed/.code  = {\tikzcdset{markwith = {\draw[-] (-.4ex,-.4ex) -- (.4ex,.4ex);}};},
  markwith/.code ={
    \pgfutil@ifundefined%
    {tikz@library@decorations.markings@loaded}%
    {\pgfutil@packageerror{tikz-cd}{You need to say %
      \string\usetikzlibrary{decorations.markings} to use arrows with markings}{}}{}%
    \pgfkeysalso{/tikz/postaction = {
      /tikz/decorate,
      /tikz/decoration={markings, mark = at position 0.5 with {#1}}}
    }
  },
}
\makeatother
\usepackage[T1]{fontenc}
\usepackage{amssymb,amsmath,latexsym,amsthm,paralist,mathrsfs,dsfont}
\usetikzlibrary{babel}
\usepackage{varioref}
\usepackage[pdfpagelabels, pdftex]{hyperref}
\hypersetup{
  pdftitle={The tame site of a scheme},
  pdfauthor={Katharina H\"{u}bner and Alexander Schmidt},
  pdfsubject={},
  pdfkeywords={},
  colorlinks=true,    
  linkcolor=black,     
  citecolor=black,     
  filecolor=black,      
  urlcolor=black,       
  breaklinks=true,
  bookmarksopen=true,
  bookmarksnumbered=true,
  pdfpagemode=UseOutlines,
  plainpages=false,
  unicode=true
  }
\usepackage[capitalise]{cleveref}
\let\noi\noindent

\newcommand{\ds}{\displaystyle}
\newcommand{\ms}{\medskip}

\newcommand{\rd}{\mathrm{red}}

\newcommand{\Aut}{\mathrm{Aut}}

\newcommand{\id}{\mathrm{id}}

\newcommand{\im}{\mathrm{im}}

\newcommand{\Gal}{\mathrm{Gal}}

\newcommand{\Hom}{\mathrm{Hom}}

\newcommand{\cont}{\mathrm{cts}}

\newcommand{\pr}{\mathit{pr}}
\newcommand{\Nis}{\mathrm{Nis}}

\newcommand{\m}{\mathfrak{m}}
\newcommand{\n}{\mathfrak{n}}
\newcommand{\p}{\mathfrak{p}}
\newcommand{\q}{\mathfrak{q}}

\newcommand{\Z}{\mathds Z}

\newcommand{\Q}{\mathds Q}

\newcommand{\F}{\mathds F}
\newcommand{\R}{\mathds R}

\newcommand{\lang}{\longrightarrow}

\renewcommand{\O}{\mathcal{O}}

\newcommand{\Cat}{\mathrm{Cat }}

\newcommand{\res}{\mathrm{res}}

\newcommand{\Zar}{\mathrm{Zar}}
\newcommand{\et}{\mathrm{et}}
\newcommand{\set}{\mathrm{set}}

\newcommand{\tr}{\mathrm{tr}}

\newcommand{\Mor}{\mathrm{Mor}}
\newcommand{\Ab}{{\mathcal{A}b}}

\newcommand{\sh}{\mathit{sh}}
\newcommand{\strh}{\mathit{strict}}

\newcommand{\A}{{\mathbb A}}

\newcommand{\nr}{{\mathit{nr}}}

\newcommand{\PrSh}{\mathit{PrSh}}
\newcommand{\Sh}{\mathit{Sh}}

\newcommand{\cH}{{\mathscr H}}

\newcommand{\cU}{{\mathscr U}}
\newcommand{\cV}{{\mathscr V}}
\newcommand{\cW}{{\mathscr W}}
\newcommand{\cX}{{\mathscr X}}
\newcommand{\cY}{{\mathscr Y}}
\newcommand{\cZ}{{\mathscr Z}}

\newcommand{\liso}{\mathrel{\hbox{$\longrightarrow$} \kern-2.4ex\lower-1ex\hbox{$\scriptstyle\sim$}\kern1.7ex}}

\newtheoremstyle{alexthm}
  {}
  {}
  {\sl }
  {}
  {\bf}
  {.}
  {.5em}
  {}
\theoremstyle{alexthm}

\newtheorem{theorem}{Theorem}[section]
\newtheorem*{theorem*}{Theorem}
\newtheorem{corollary}[theorem]{Corollary}
\newtheorem{proposition}[theorem]{Proposition}
\newtheorem{lemma}[theorem]{Lemma}
\newtheorem*{lemma*}{Lemma}

\newtheorem{conjecture}[theorem]{Conjecture}

\newtheoremstyle{alexdef}
  {}
  {}
  {\rm }
  {}
  {\bf}
  {.}
  {.5em}
  {}
\theoremstyle{alexdef}
\newtheorem*{example*}{Example}
\newtheorem{example}[theorem]{Example}

\newtheorem{remark}[theorem]{Remark}
\newtheorem{remarks}[theorem]{Remarks}
\newtheorem{definition}[theorem]{Definition}

\DeclareMathOperator{\Spec}{\textit{Spec}}
\DeclareMathOperator{\Specm}{\textit{Specm}}

\DeclareMathOperator{\RZ}{\textit{RZ}}
\DeclareMathOperator{\supp}{\textrm{supp}}
\newcommand{\sets}{\textrm{Sets}}

\DeclareMathOperator{\ch}{char}
\DeclareMathOperator{\Val}{Val}
\newcommand{\Sch}{\mathrm{Sch}}
\newcommand{\Cor}{\mathrm{Cor}}

\DeclareMathOperator*{\colim}{colim}
\DeclareMathOperator{\Spa}{\mathit{Spa}}

\title{The tame site of a scheme}
\author{Katharina H\"{u}bner and Alexander Schmidt}
\date{}
\begin{document}
\maketitle
\tableofcontents

\section{Introduction}

\noindent
The \'{e}tale fundamental group of a scheme plays a similar role in algebraic geometry as the topological fundamental group in algebraic topology.  For a scheme $X$ of characteristic $p>0$ however, the $p$-part of $\pi_1^\et(X)$ is not well-behaved, e.g., it is not ($\A^1$-)homotopy invariant. Therefore the \emph{tame} fundamental group has been studied in positive and mixed characteristic (Grothendieck-Murre \cite{GM71}, Kerz-Schmidt \cite{KeSch10}). Unfortunately, lacking an associated tame cohomology theory, sometimes ad hoc arguments have to be used in applications.
It would be helpful to have a tame Grothendieck topology whose associated fundamental group is the tame fundamental group.
Such a topology would provide a tame cohomology theory and, in addition, higher tame homotopy groups.
The latter seem to be even more desirable because the higher \'{e}tale homotopy groups of affine varieties vanish in positive characteristic (Achinger \cite{Ach17}). In this paper we give a definition of a tame site that hopefully will prove useful as a tool to separate tame from wild phenomena in scheme theory.

\medskip Let $X$ be a scheme over  a fixed base scheme $S$. The tame site $(X/S)_t$ which will be defined below has the same underlying category as the small \'{e}tale site $X_\et$  but less coverings. Every Nisnevich covering is tame, so we have natural morphisms of sites
\[
 X_\et \stackrel{\alpha}{\lang} (X/S)_t \stackrel{\beta}{\lang} X_\Nis.
\]
Tameness of a covering should be thought of as ``at most tamely ramified along the boundary of compactifications over $S$''. We will use the associated valuation space as a convenient technical tool to make this precise.

\medskip
A technical advantage of the tame site sitting between the \'{e}tale and the Nisnevich site is that motivic techniques like sheaves with transfers \`{a} la Voevodsky and the $\A^1$-homotopy category of Morel-Voevodsky work essentially without change also for the tame site. A technical disadvantage is that the topology is not local enough in the sense that we cannot sufficiently separate the valuations  that occur in the definition of tameness. We will resolve this problem by comparing the tame cohomology of an $S$-scheme $X$ with the tame cohomology of the adic space $\Spa(X,S)$ as defined in H\"{u}bner \cite{HueAd}.

\bigskip
Our wish list of properties one would expect of an ``ideal'' tame site  $(X/ S)_t$
contains the following:

\begin{enumerate}
\item (Topological invariance) If $X'\rightarrow X$ is a universal homeomorphism of $S$-schemes, then the sites $(X/S)_t$ and $(X'/S)_t$ are isomorphic.
\item (Comparison with \'{e}tale cohomology for invertible coefficients) Let $F\in \Sh_\et(X)$ be an abelian sheaf  with $mF=0$ for some $m$ which is invertible on $S$. Then the natural morphism $\alpha: X_\et \to (X/S)_t$ induces isomorphisms
\[
H^n_t(X/S,\alpha_*F)\cong  H^n_\et(X,F) \quad \text{for all }n\ge 0.
\]

\item (Comparison with \'{e}tale cohomology for proper schemes) Let $F\in \Sh_\et(X)$ be an abelian sheaf and assume that $X/S$ is proper. Then
\[
H^n_t(X/S,\alpha_*F)\cong H^n_\et(X,F) \quad \text{for all }n\ge 0.
\]
  \item (Homotopy invariance of cohomology): The projection $p: \A^1_X\to X$ induces isomorphisms
 \[
 H^n_t(X/S,F) \liso H^n_t(\A^1_X/S,p^*F)
 \]
 for all $n\ge 0$ and all torsion sheaves $F\in \Sh_t(X/S)$.
 \item (Excision) Let $\pi: X' \rightarrow X$ be an \'{e}tale morphism and
$Z \hookrightarrow X$, $Z' \hookrightarrow X'$ closed immersions with open complements $U=X\smallsetminus Z$, $U'=X'\smallsetminus Z'$. Assume that $\pi(U') \subset U$
and that $\pi$ induces an isomorphism $Z'_\rd \liso Z_\rd $.
Then the induced homomorphisms between tame cohomology groups with supports
\[
H^n_{t,Z}(X/S, F) \lang H^n_{t,Z'}(X'/S, \pi^*F)
\]
are isomorphisms for every $F \in \Sh_t(X/S)$ and all $n \geq 0$.
  \item (Purity, locally constant coefficients) If $X$ is regular and of finite type over~$S$, $S$ is pure of characteristic $p$>0, $U\subset X$ is a dense open subscheme and $F$ is a locally constant $p$-torsion sheaf, then
  \[
  H^n_t(X/S, F) \cong H^n_t(U/S, F|_U).
  \]
 for all $n\ge 0$.
\end{enumerate}

Moreover, we expect tame versions of the proper and the smooth base change theorem and purity for deRham-Witt sheaves. In addition, we expect finiteness properties of tame cohomology groups and cohomological dimension in those cases where they hold for the \'{e}tale cohomology (e.g., schemes of finite type over $\Z$ or over separably closed fields).

\medskip
In this paper we will prove properties (1), (2), (3) and (5) in general. Assuming resolution of singularities (in the version of  \cref{ros} below), we prove (6) and (4) for locally constant coefficients. If $X$ is locally noetherian, we show that the (curve-)tame fundamental group as defined in \cite{KeSch10} occurs as the natural fundamental group of the site $(X/S)_t$. For schemes of dimension less or equal to one, we prove finiteness properties in \cref{sec:finite}. Purity for deRham-Witt sheaves on curves has been addressed in \cite{HueCur}.

\medskip
The essential point in proving (4) and (6) is the comparison theorem \ref{main compare} between the tame cohomology of an $S$-scheme $X$ and the cohomology of the associated adic space $\Spa(X,S)$. At the moment, we can prove this only in pure characteristic. The method is to compare the associated \v{C}ech cohomologies. To prove that the corresponding sheaf and the \v{C}ech cohomologies coincide, we adapt M.~Artin's method of joins of hensel rings \cite{Art71} to the respective situations. In the adic case this is rather involved and requires a careful analysis of relative Riemann-Zariski spaces as introduced by M.~Temkin \cite{Tem11}.

\medskip
Finally, in \cref{sec_suslin} we give an application to Suslin homology which in fact was the initial motivation for our construction of the tame site.
We fix an algebraically closed field~$k$ and assume that resolution of singularities holds over~$k$.
For a scheme $X$ of finite type over  $k$, we construct natural maps from tame to Suslin cohomology
\[
\beta_n: H^n_t(X,\Z/m\Z)\lang H^n_S(X,\Z/m\Z), \ n\ge 0, \ m\ge 1.
\]
We conjecture that all $\beta_n$ are isomorphisms. This extends work of Suslin-Voevodsky \cite{SV96} in the case $(m,\ch(k))=1$ and of Geisser-Schmidt \cite{GS16} for $n=1$ and general $m$.

\medskip\noindent
\emph{Acknowledgement:} The authors want to thank M.~Temkin for helpful discussions on Riemann-Zariski points. The construction of the map between tame and Suslin cohomology arose from discussions of the second author with T.~Geisser while writing  \cite{GS16}. We would like to thank the referee, whose constructive criticism has helped to improve the presentation.

\section{Definition of the tame site}

By a valuation on a field $K$ we mean a non-archimedean valuation, not necessarily discrete or of finite rank. The trivial valuation is included. If $v$ is a valuation on $K$, we denote by $\O_v,\m_v$ and $k(v)$ the valuation ring, its maximal ideal and the residue field. By $\O_v^h$ and $\O_v^\sh$ we denote the henselization and strict henselization of $\O_v$ and by $K_v^h$ and $K_v^\sh$ their quotient fields.

Let $v$ be a valuation on $K$ and $w$  an extension of $v$ to a finite separable extension field $L/K$. We call $w/v$ \emph{unramified} if $\O_v \to \O_w$ is \'{e}tale, i.e., $L_w^\sh =K_v^\sh$,  and \emph{tamely ramified} if the field extension $L_w^\sh /K_v^\sh$ is of degree prime to the residue characteristic $p=\ch k(v)$. In this case $L_w^\sh /K_v^\sh$ is automatically Galois with abelian Galois group of order prime to $p$. If $L/K$ is Galois, then $w/v$ is unramified (resp.\ tamely ramified) if and only if the inertia group $T_w (L/K)$ (resp.\ the ramification group $R_w(L/K)$) is trivial. (See \cite{Ray70} and \cite{EP2005}.)

\begin{definition}
An \emph{$S$-valuation} on an $S$-scheme $X$ is a valuation $v$ on the residue field $k(x)$ of some point $x\in X$ such that there exists a morphism
$
\varphi: \Spec(\O_v) \to S
$
making the diagram
\[
\begin{tikzcd}
\Spec(k(x)) \arrow[rr] \arrow[d] &  &X \arrow[d,"f"] \\
\Spec(\O_v) \arrow[rr, "\varphi"] &  & S
\end{tikzcd}
\]
commutative (if $S$ is separated, $\varphi$ is unique if it exists).
The set of all $S$-valuations is denoted by $\Val_S X$.

We denote elements of $\Val_S X$ in the form $(x,v)$, $x\in X$, $v\in \Val_S(k(x))$.
\end{definition}

\begin{remarks} \begin{enumerate}
                  \item Directly by definition, we have a disjoint union decomposition
                  \[
\Val_S X= \coprod_{x\in X} \Val_S(k(x)).
\]
                  \item If $S=\Spec k$ for a field $k$, then $S$-valuations are the valuations with trivial restriction to~$k$.
                  \item If $S=\Spec \Z$, then every valuation is an $S$-valuation.
                \end{enumerate}
\end{remarks}

\noindent
Given a commutative square of scheme morphisms
\[
\begin{tikzcd}
X'\arrow[d,"f'"']\rar{\varphi} &X\dar{f}\\
S' \rar{\psi} & S,
\end{tikzcd}
\]
there is an associated map
\[
\Val_\psi \varphi: \Val_{S'} X' \lang \Val_S X,\quad  (x',v') \longmapsto (x=\varphi(x'), v=v'|_{k(x)}).
\]

\begin{definition}
The \emph{tame site} $(X/S)_t$ consists of the following data:

\smallskip\noindent
The category $\Cat (X/S)_t$ is the category of \'{e}tale morphisms $p:U\to X$.

\smallskip\noindent
A family $(U_i \to U)_{i\in I} $ of morphisms in $\Cat (X/S)_t$ is a covering if it is an \'{e}tale covering and for every point $(u,v)\in \Val_SU$ there
exists an index~$i$ and a point $(u_i,v_i)\in \Val_S U_i$ mapping to $(u,v)$ such that $v_i/v$ is  tamely ramified.
\end{definition}
We will use the notation $\Sh_t^\sets (X/S)$ resp.\ $\Sh_t(X/S)$ for the category of sheaves of sets resp.\ abelian groups on $(X/S)_t$.

\begin{remark}
There are obvious morphisms of sites
\[
X_\et \stackrel{\alpha}{\lang} (X/S)_t \stackrel{\beta}{\lang} X_\Nis.
\]
In particular, every \'{e}tale sheaf is a tame sheaf. We mention the following special cases:

 For an abelian group $A$, the presheaf (constant Zariski-sheaf)
\[
U \longmapsto {\underline A} (U):=\mathrm{Map}_\cont (U,A)
\]
is a tame sheaf.

For a quasi-coherent sheaf $Q$ on $X_\Zar$, the presheaf $W(Q)$ given by
\[
W(Q)(U\stackrel{h}{\to} X):= \big(h_\Zar^*(Q)  \otimes_{h^*_\Zar (\O_X)} \O_U\big)  (U)
\]
is a sheaf on $(X/S)_t$.
\end{remark}

Consider a commutative diagram
\[
\begin{tikzcd}
X'\dar\rar &X\dar\\
S' \rar & S
\end{tikzcd}
\]
of scheme morphisms. Then the pullback $(U\to X)\mapsto (U'=U\times_XX'\to X')$ induces  a morphism of sites $\varphi: (X'/S')_t \to (X/S)_t$.

\begin{lemma}\label{inverse-image}
The inverse image functors $\varphi^*_\sets: \Sh_t^\sets(X/S)\to \Sh_t^\sets(X'/S')$ and  $\varphi^*: \Sh_t(X/S)\to \Sh_t(X'/S')$ are exact. The direct image functor $\varphi_*:  \Sh_t(X'/S')\to \Sh_t(X/S)$ sends injective abelian  sheaves to injective abelian sheaves.
\end{lemma}

\begin{proof}
The underlying categories of the tame sites are the same as for the small \'{e}tale sites. Hence the first statement follows in the same way as for the \'{e}tale topology from the fact that the index categories for the colimits which define the presheaf-pullback are filtered (cf.\ \cite[II,\,Thm.\,4.14]{Art62}). The second assertion is a formal consequence of the adjunction $\varphi^*\dashv \varphi_*$.
\end{proof}

\begin{lemma}\label{proper-lem}
Let $X \stackrel{f}{\lang}S \stackrel{g}{\lang}T$
be scheme morphisms. If $g$ is proper, then the natural morphism of sites
\[
(X/S)_t \lang (X/T)_t
\]
is an isomorphism.
\end{lemma}

\begin{proof}
If $g: S\to T$ is proper, then the valuative criterion for properness implies $\Val_S U=\Val_T U$ for every $p: U\to X$ \'{e}tale. Whence the statement.
\end{proof}
\begin{definition}\label{deftamept}
A \emph{tame point} of $(X/S)_t$ is a pair $(\bar{x},\bar{v})$ where $\bar{x}:\Spec k(\bar{x})\to X$ is a morphism from the spectrum of a field to~$X$ and~$\bar{v}$ is an $S$-valuation on $k(\bar{x})$ such that~$k(\bar{x})$ does not admit a nontrivial,  finite, separable extension to which $\bar{v}$ has a tamely ramified extension. In other words, $k(\bar{x})$ is strictly henselian with respect to~$\bar{v}$ and the absolute Galois group of $k(\bar x)$ is a pro-$p$ group, where $p   $ is the residue characteristic of $\bar{v}$ (if $p=0$ this means that the group is trivial).

If $\bar v$ is the trivial valuation, we call the point $(\bar{x},\bar{v})$ \emph{trivial point}. Note that (regardless of the residue characteristic) the field $k(\bar x)$ of a trivial point is separably closed.
\end{definition}

A tame point $(\bar{x},\bar{v})$ induces a morphism of sites
\[
\bar x: (\Spec k(\bar{x})/\Spec \O_{\bar v})_t \lang (X/S)_t.
\]
By \cref{inverse-image}, the inverse image sheaf functor ${\bar x}^*$ is exact. Since every tame covering of $\Spec(k(\bar x))$ splits, the global sections functor on $(\Spec k(\bar{x})/\Spec \O_{\bar v})_t$ is exact. Therefore the functor ``stalk at $\bar x$''
\[
 F \longmapsto F_{\bar x}:=\Gamma (\Spec k(\bar{x}),\bar x^*F),
\]
is a topos-theoretical point of $(X/S)_t$. The trivial points correspond to the usual geometric points of the \'{e}tale site (followed by $\alpha: X_\et \to (X/S)_t$).

\begin{remark}\label{stalkremark} The stalk of the structure sheaf at a point of the Nisnevich resp.\ \'{e}tale site is the henselization resp.\ the strict henselization of the local ring of the underlying scheme theoretic point.

Let $(\bar x,\bar v)$ be a tame point of $X/S$ with underlying scheme theoretic point $x_0\in X$. Let $k(\bar x)_s/k(x_0)$ be the relative separable algebraic closure of $k(x_0)$ in $k(\bar x)$ and choose a separable algebraic closure  $k^s/k(\bar x)_s$. Denote the strict henselization of $\O_{X,x_0}$ with respect to $k^s /k(x_0)$ by $\O_{X,\bar x}^{sh}$. Then the stalk $\O_{X, (\bar x,\bar v)}$ of the structure sheaf at $(\bar x,\bar v)$ is the unique ind-\'{e}tale subextension of
\[
\O_{X,x_0}^{h} \hookrightarrow \O_{X,\bar x}^{sh}
\]
which corresponds to the subextension $k(\bar x)_s$ of $k^s/k(x_0)$.
\end{remark}

\begin{definition}\label{tame-henselization-def}  We call the ring $\O_{X, (\bar x,\bar v)}$ described in \cref{stalkremark} the \emph{tame henselization} of $X$ at the point $(\bar x,\bar v)$.
\end{definition}

\begin{lemma}\label{enoughpoints}
The site $(X/S)_t$ has enough points.
\end{lemma}

\begin{proof}
We fix a separable closure $k(x)^s$ of $k(x)$ for every point $x\in X$. For each $S$-valuation $v$ on $k(x)$ we choose a prolongation $\bar v$ to $k(x)^s$ and consider the fixed field $K_v:=(k(x)^s)^{R_{\bar v}}$ of its ramification group. Then the tame points $(\bar x_L, \bar v|_L)$ with
\[
\bar x_L: \Spec L \lang \Spec k(x) \lang X,
\]
where $x$ runs through the points of $X$, $v$ runs through the $S$-valuations of $k(x)$ and $L$ runs through the finite subextensions of $k(x)^s/K_v$ establish a conservative set of points. Indeed, this follows immediately from the definition of coverings in $(X/S)_t$. \end{proof}

We learn that exactness of sequences of tame sheaves can be checked stalkwise at tame points. As in the \'{e}tale case, we obtain

\begin{lemma}
Let $\pi: X' \to X$ be a finite morphism between $S$-schemes. Then
\[
\pi_*: \Sh_t(X'/S) \to \Sh_t(X/S)
\]
is exact. For $F' \in \Sh_t(X'/S)$ we have $H^q_t(X/S,\pi_* F')\cong H^q_t(X'/S,F')$ for all $q\ge 0$.
\end{lemma}

\begin{proof}
The proof is similar to that for the \'{e}tale topology:
The first statement follows by considering the stalks of $\pi_*F'$ and using the fact that a finite algebra over a local henselian ring is a finite product of local henselian rings. From this we obtain $R^q\pi_* F'=0$ for $q>0$ and the second statement follows from the Leray-Serre spectral sequence.
\end{proof}

\section{Topological invariance and excision}
We start by collecting properties of the tame site which follow from or in a similar way as the corresponding properties of the small \'{e}tale site.

\begin{proposition}[Topological invariance] \label{topinvariance}
Let $f: X_0\to X$ be a universal homeomorphism of $S$-schemes. Then the tame sites
$(X_0/S)_t$ and $(X/S)_t$ are isomorphic. In particular, there are isomorphisms
\[
H^n_t(X/S,F) \liso H^n_t(X_0/S,f^*F)
\]
for all $n\ge 0$ and $F\in \Sh_t(X/S)$.
\end{proposition}

\begin{proof}
By the topological invariance of the small \'{e}tale site \cite[Theorem 18.1.2]{EGAIV.4}, $Y\mapsto Y\times_XX_0$ defines an equivalence $f^{-1}: \Cat (X/S)_t \liso \Cat (X_0/S)_t$ of the underlying categories. By \cite[TAG 04DF]{stacks-project}, $f$ is integral. Hence, for any point $x_0 \in X_0$ with image $x\in X$, the induced field extension $k(x_0)/k(x)$ is algebraic and a universal homeomorphism, hence purely inseparable.
This implies that every $S$-valuation on $k(x)$ has a unique extension to $k(x_0)$. By \cite[Lemma 1.1]{quasi-purity}, for a valuation $v_0$ on $k(x_0)$ with restriction $v$ to $k(x)$, the induced homomorphism on ramification groups is an isomorphism. Hence, $f^{-1}$ respects tame coverings. The result follows.
\end{proof}

Let $X\to S$ be an $S$-scheme, $i:Z\hookrightarrow X$ a closed immersion and $U=X\smallsetminus Z$ the open complement.
The right derivatives of the left exact functor ``sections with support in $Z$''
\[
F \mapsto \ker (F(X) \to F(U))
\]
are called the \emph{tame cohomology groups with support in $Z$}. Notation:  $H^*_{t,Z}(X/S, F)$.

Exactly in the same way as in the \'{e}tale case (\cite[III,\,(2.11)]{Art62}), one shows the existence of a long
 exact sequence
\begin{equation*}\label{long exact for support}
  \ldots \to  H^p_{t,Z}(X/S, F) \rightarrow H^p_t (X/S, F) \rightarrow H^p_t(U/S,F) \rightarrow H^{p+1}_{t,Z}(X/S, F) \rightarrow \ldots
\end{equation*}

\begin{proposition}[Excision] Let $\pi: X' \rightarrow X$ be a morphism of $S$-schemes,
$Z \hookrightarrow X$, $Z' \hookrightarrow X'$ closed immersions and $U=X\smallsetminus Z$, $U'=X'\smallsetminus Z'$ the open complements. Assume that

\ms
\begin{compactitem}
\item $\pi: X' \to X$ is \'{e}tale,
\item $\pi$ induces an isomorphism $Z'_\rd \liso Z_\rd $,
\item $\pi(U') \subset U$.
\end{compactitem}

\ms\noi
Then the induced homomorphism
\[
H^p_{t,Z}(X/S, F) \liso H^p_{t,Z'}(X'/S, \pi^*F)
\]
is an isomorphism for every sheaf $F \in \Sh_t(X/S)$ and all $p \geq 0$.
\end{proposition}
\begin{proof} The standard proof for the \'{e}tale topology applies:
Since $\pi$ belongs to $\Cat (X/S)_t$, $\pi^*$ has the exact left adjoint ``extension by zero''.  Hence $\pi^*$ sends injectives to injectives. Since $\pi^*$ is exact,  it suffices to deal with the case $p =0$. Without changing the statement, we can replace all occurring schemes by their reductions (see \cref{topinvariance}).
By assumption, $X' \sqcup U \rightarrow X$, is a tame covering.
For $\alpha \in H^0_{t,Z} (X, F)$ mapping to zero in  $H^0_{t,Z'}(X',\pi^*F)$ we therefore obtain $\alpha=0$.

Now let $\alpha' \in H^0_{t,Z'}(X', \pi^*F)$ be given.
We show that $\alpha'$ and $0\in H^0_t(U, F)$ glue to an element in $H^0_{t,Z}(X, F)$. The only nontrivial compatibility on intersections is  $p^*_1(\alpha')=p^*_2(\alpha')$ for $p_1,p_2:  X'\times_XX' \to X'$.
This can be checked on stalks noting that $Z'\liso Z$ implies that the two projections $Z'\times_Z Z'\to Z'$ are the same.
\end{proof}

\section{Continuity}\label{sec:cont}

The aim of this section is to prove theorems \ref{colimitsheaves} and \ref{inverselimit} below, which deal with the cohomology of colimits and limits.
We start with the following finiteness result:

\begin{theorem}\label{finite-cov}
Assume that $S$ is quasi-compact and quasi-separated and let $X$ be a quasi-compact $S$-scheme. Then every tame covering of $X$ admits a finite subcovering.
\end{theorem}

For the proof we need some preparations.
The definition of the tame site suggests to tackle the problem by endowing $\Val_S X$ with a suitable topology and then to apply a compactness argument.
A convenient way to do this is to move the problem to the setting of adic spaces, where it was solved in \cite{HueAd}.

\medskip
Following M.~Temkin \cite{Tem11}, we associate with a morphism of schemes $f: X \to S$ the set $\Spa(X,S)$, whose elements are triples $(x,v,\varepsilon)$, where~$x$ is a point of~$X$,~$v$ is a valuation of~$k(x)$ and $\varepsilon : \Spec \O_v \to S$ is a morphism compatible with $\Spec k(x) \to S$. If $S$ is separated, then $\varepsilon$ is uniquely determined (if it exists) by $(x,v)$. In other words, we have a natural forgetful map
\[
\Spa(X,S) \lang \Val_S X,
\]
which is bijective if $S$ is separated.
Here, $\Val_S X$ is the earlier defined space of pairs $(x,v)$ with $x \in X$ and~$v$ a valuation of~$k(x)$ with center in S (but we do not specify a homomorphism $\Spec \O_v \to S$).

\begin{remark}
 Note that the space $\Val_X(S)$ defined in \cite{Tem11} is not the same as our $\Val_S X$.
 Temkin's $\Val_X(S)$ denotes the relative Riemann-Zariski space of~$X$ over~$S$.
 It can be viewed as the subspace of all Riemann-Zariski points (see \cref{RZ-points} below) of $\Spa(X,S)$.
\end{remark}

If $X = \Spec A$ and $S = \Spec R$ are affine, then $\Spa(X,S)$ coincides with the set of points of R.~Huber's affinoid adic space $\Spa(A,A^+)$, where $A^+$ ist the integral closure of $R$ in $A$ and $A$ and $A^+$ are equipped with the discrete topology, cf.\ \cite{Hu96}. By patching, we obtain on $\Spa(X,S)$ the structure of a discretely ringed adic space.
If
 \[
  \begin{tikzcd}
  X'\dar\rar{\varphi} &X\dar\\
  S' \rar & S
  \end{tikzcd}
 \]
is a commutative diagram of scheme morphisms, then
 \[
  \Spa(\varphi) : \Spa(X',S') \lang \Spa(X, S)
 \]
is a morphism of adic spaces, in particular it is continuous.
The topology of $\Spa(X,S)$ is generated by the images of $\Spa(X',S')$ in $\Spa(X,S)$ coming from commutative diagrams as above
with~$X'$ and~$S'$ affine, $X' \to X$ an open immersion and $S' \to S$ locally of finite type.

\medskip
Recall that a topological space is called \emph{spectral} if it is homeomorphic to $\Spec(R)$ for some commutative ring $R$.

\begin{lemma}\label{spectral}
If~$S$ and~$X$ are quasi-compact and quasi-separated, then $\Spa(X,S)$ is a spectral space. In particular,
$\Spa(X,S)$ is a quasi-compact and quasi-separated topological space.
\end{lemma}

\begin{proof}
By~\cite[Corollary\,to\,Proposition\,7]{Ho69}, we have to show that $\Spa(X,S)$ is locally spectral, quasi-compact and quasi-separated.
 Let
 \[
  S = \bigcup_i S_i \qquad \text{and} \qquad X \times_S S_i = \bigcup_{j_i} X_{j_i}
 \]
 be finite coverings by open, affine subschemes.
 Then
 \[
  \Spa (X,S) = \bigcup_{i,j_i} \Spa(X_{j_i},{S_i}).
 \]
The spaces $\Spa(X_{j_i},S_i)$ are  spectral  by \cite[Theorem~3.5]{Hu93}.
Hence,  $\Spa(X,S)$ is locally spectral and quasi-compact.
In order to show quasi-separateness, fix~$i,i'$ and $j_i,j'_{i'}$.
The intersection
 \[
  \Spa(X_{j_i},{S_i}) \cap \Spa(X_{j'_{i'}},{S_{i'}}) = \Spa((X_{j_i} \cap X_{j'_{i'}}),{S_i \cap S_{i'}})
 \]
is quasi-compact by what we just saw and the quasi-separateness assumptions on~$S$ and~$X$.
Therefore, $\Spa(X,S)$ is quasi-separated.
\end{proof}

\medskip
If $\varphi:Y \to X$ is an \'{e}tale morphism of $S$-schemes, then $\Spa(\varphi): \Spa(Y,S)\to \Spa(X,S)$ is an \'{e}tale morphism of adic spaces in the sense of \cite[Definition~1.6.5]{Hu96}. By definition \cite[Definition 3.1]{HueAd}, $\Spa(\varphi)$ is \emph{tame} if for every point $(y,v,\varepsilon)$ of $\Spa(Y,S)$, the valuation $v$ is (at most) tamely ramified in the finite separable field extension $k(y)/k(\varphi(y))$. In particular, tameness at a point of $\Spa(Y,S)$ only depends on its image in $\Val_S Y$. The following proposition is essential for us:

\begin{proposition}[Openness of the tame locus, {\cite[Corollary\,4.4]{HueAd}}]\label{opennness-tame-locus}
Let $\varphi:Y \to X$ be an \'{e}tale morphism of $S$-schemes. Then the set of points where the morphism of adic spaces $\Spa(\varphi): \Spa(Y,S) \to \Spa(X,S)$ is tame is an open subset of $\Spa(Y,S)$.
\end{proposition}

\begin{proof}[Proof of \cref{finite-cov}]
Since $X$ admits a finite Zariski-covering by affine schemes, we may assume that $X$ is affine. Let $(U_i\to X)$ be a tame covering.
 By \cref{opennness-tame-locus}, the tame locus of $\Spa(U_i,S)\to \Spa(X,S)$ is an open subspace $\cU_i$ of $\Spa(U_i,S)$ for all~$i$. By \cite[1.10.12.\,i)]{Hu96}, \'{e}tale morphisms of adic spaces are open, hence the images  $\cW_i\subset \Spa(X,S) $ of the $\cU_i$ are open. As $(U_i \to X)$ is a tame covering, the $\cW_i$ cover $\Spa(X,S)$.  Since $X$ is affine, $\Spa(X,S)$ is a spectral space by \cref{spectral}. In particular, $\Spa(X,S)$ is quasi-compact.

We conclude that there is a finite subset $J\subset I$ with $\Spa(X,S)=\bigcup_{j\in J} \cW_j$. Therefore, $(U_j \to W)_{j\in J}$ is a finite subcovering of $(U_i \to X)_{i\in I}$.
\end{proof}
As in \cite[VII, 3.2]{SGA4} for the \'{e}tale site, we define the restricted tame site
\[
(X/S)_t^\res
\]
as the restriction of $(X/S)_t$ to the subcategory of all $U\in \Cat(X/S)_t$ such that $U\to X$ is of finite presentation.

Assume that $X$ is quasi-compact and quasi-separated. Then the same is true for any such $U$ and \cref{finite-cov} shows that the restricted site is noetherian. Moreover, the categories of sheaves on $(X/S)$ and $(X/S)_t^\res$ are naturally equivalent. Hence the same argument as in the \'{e}tale case \cite[VII,\,Proposition\,3.3]{SGA4} shows

\begin{theorem} \label{colimitsheaves} Assume that $S$ and $X$ are quasi-compact and quasi-separated and let $(F_i)$ be a filtered direct system of abelian sheaves on $(X/S)_t$. Then
\[
\colim_i \, H^q_t(X/S, F_i)\cong H^q_t(X/S, \colim_i \, F_i)
\]
for all $q\ge 0$.
\end{theorem}

Now we consider limits of schemes. The main result is:

\begin{theorem} \label{inverselimit}
Let~$(f_i:X_i \to S_i)_{i\in I}$ be a filtered inverse system of scheme morphisms with affine transition morphisms and assume that all $S_i$ and $X_i$ are quasi-compact and quasi-separated.
Denote by $X \to S$ its inverse limit. Then, for every \'{e}tale surjection
$U \to X$ of finite presentation which is a tame covering with respect to $S$, there exists an index $i\in I$ and an \'{e}tale surjection
$U_i \to X_i$ of finite presentation which is a tame covering with respect to $S_i$ and whose base change to $X$ is $U\to X$.

The natural map from $(X/S)_t^\res$ to the limit topology of the $(X_i/S_i)_t^\res$ is an isomorphism.
\end{theorem}

\begin{proof}
By~\cite[VII, 5.6]{SGA4}, there is $i \in I$ such that $U \to X$ is the pullback of an \'{e}tale surjection $U_i \to X_i$.  For every $j \geq i$ denote by $U_j \to X_j$ its pullback to~$X_j$. We have to show that there is $j \geq i$ such that $U_j \to X_j$ is tame over~$S_j$. By \cref{opennness-tame-locus}, the set of points in $\Spa(U_j,S_j) $ where
$
 \Spa(U_j,S_j) \to \Spa(X_j, S_j)
$
is tame is open.  Denote by~$\cZ_j\subset \Spa(X_j, S_j)$ the complement of the image of this set. It consist of all points $(x_j,v_j,\varepsilon_j)\in \Spa(X_j, S_j)$ such that there is no point  $(u_j, w_j, \mu_j) \in \Spa(U_j,S_j)$ lying over $(x_j,v_j,\varepsilon_j)$ such that $w_j/v_j$ is tamely ramified in $k(u_j)/k(v_j)$. Since \'{e}tale morphisms of adic spaces are open, $\cZ_j$ is a closed subset of $\Spa(X_j,S_j)$. Since tameness is stable under base change, for $k \geq j \geq i$ the transition map
 $
 \Spa(X_k,S_k) \to \Spa(X_j,S_j)
 $
sends $\cZ_k$ to $\cZ_j$.

For $(u,w,\mu) \in \Spa(U,S)$ with images $(u_j,w_j,\mu_j) \in \Spa(U_j,S_j)$ lying over $(x,v,\varepsilon) \in \Spa(X,S)$ with images $(x_j,v_j,\varepsilon_j) \in \Spa(X_j,S_j)$ we have $k(u)=\bigcup k(u_j)$ and $k(x)=\bigcup k(x_j)$. For sufficiently large $j$, we have $k(u)=k(u_j)k(x)$  (compositum in $k(u)$).

Let $K$ be the Galois closure of $k(u)/k(x)$ in some separable closure of $k(u)^s$ of $k(u)$ and $K_j$ the Galois closure of $k(u_j)/k(x_j)$ in $k(u)^s$. Then $K=\bigcup K_j$ and $G(K/k(x)) \to G(K_j/k(x_j))$ is an isomorphism for sufficiently large~$j$. Let $w'$ be a valuation on $K$ with $w'|_{k(u)}=w$ and let $w_j'=w'|_{K_j}$, so that $w_j'|_{k(u_j)}=w_j$. Then the induced homomorphism of ramification groups $R_{w'}(K/k(x)) \to R_{w_j'}(K_j/k(x_j))$ is an isomorphism for sufficiently large $j$. If $w/v$ is tame, then $R_{w'}(K/k(x))\subset G(K/k(u))$, hence $R_{w_j'}(K_j/k(x_j))\subset G(K_j/k(u_j))$ and $w_j/v_j$ is tame for sufficiently large $j$.
As $U \to X$ is a tame covering, we conclude that
$$
 \lim_{j \geq i} \cZ_j = \varnothing.
$$
By \cref{spectral}, $\Spa(X_j,S_j)$ is spectral and hence by \cite[TAG 0901]{stacks-project} compact in its constructible topology. In particular, the $\cZ_j$ are compact in the constructible topology. Since the inverse limit of nonempty compact spaces is nonempty, we conclude that
$\cZ_j = \varnothing$ for large enough $j$.  In other words $U_j \to X_j$ is a tame covering. This shows the first statement.

For the second statement note that every tame covering in the limit topology of the $(X_i/S_i)_t$ defines a covering in $(X/S)_t$. To show the equivalence, it suffices to show that every morphisms
$$
U \to V
$$
of \'{e}tale $X$-schemes of finite presentation which is a tame covering with respect to $S$ comes by base change from a finite level. This follows from the first statement
\end{proof}

As an immediate consequence we obtain:

\begin{theorem} \label{cohomologyoflimit}
 In the situation of \cref{inverselimit} assume that~$i_0 \in I$ is a final object.
 Let~$F_0$ be a sheaf of abelian groups on $(X_{i_0}/S_{i_0})_t$.
 For $i \in I$ denote by~$F_i$ its pullback to~$(X_i/S_i)_t$ and by~$F$ its pullback to~$(X/S)_t$.
 Then the natural map
 $$
 \colim_{i \in I}\, H^n_t(X_i/S_i,F_i) \longrightarrow H^n_t(X/S,F)
 $$
 is an isomorphism for all $n \geq 0$.
\end{theorem}

\section{The tame fundamental group}

In this section we assume that $X$ is locally noetherian. Then the small \'{e}tale site of $X$ is locally connected \cite[Prop.\,9.5]{AM} and, having  the same underlying category, the same is true for the site $(X/S)_t$.
Therefore by \cite[\S9]{AM}, after choosing a tame base point $(\bar x,\bar v)$, we have the tame homotopy (pro)groups
\[
\pi_n^t(X/S,(\bar x,\bar v)).
\]
We call the first tame homotopy group $\pi_1^t(X/S,(\bar x,\bar v))$ the \emph{tame fundamental group}. If $X$ is connected, then by \cite[Corollary 10.7]{AM}, it classifies pointed torsors for the tame topology. Moreover, by \cite[Corollary 10.8]{AM},  there is a natural equivalence between the category of locally constant abelian tame sheaves and the category of abelian groups $A$ with a homomorphism (of pro-groups) $\pi_1^t(X/S,(\bar x,\bar v)) \to \Aut (A)$.

For a group $G$, the set of isomorphism classes of tame $G$-torsors is a subset of the set of isomorphism classes of \'{e}tale $G$-torsors. Hence, for a trivial point $\bar x=(\bar x,v_\mathit{triv})$, the morphism of pointed sites
\[
(X,\bar x)_\et \lang (X/S,\bar x)_t
\]
induces a surjection $\pi_1^\et(X,\bar x)\twoheadrightarrow \pi_1^t(X/S,\bar x)$ (here $\pi_1^\et$ denotes the \'{e}tale fundamental group of \cite{AM}, which coincides with the \emph{pro-groupe fondamentale \'{e}largi} of \cite[X,\,\S6]{SGA3}). We conclude that the finite quotients of $\pi_1^t(X/S,\bar x)$ correspond to pointed, finite, connected \'{e}tale Galois covers $f:(Y,\bar y)\to (X,\bar x)$ that are coverings for the tame site. Since the Galois group acts transitively on the fibres, this means that for \emph{every} $y\in Y$ and $w\in \Val_S k(y)$ the extension $k(y)/k(f(y))$ is at most tamely ramified at $w$. We will call such covers \emph{tame Galois covers}.  If $X$ is geometrically unibranch, then the group $\pi_1^\et(X,\bar x)$ is profinite by \cite[Theorem\,11.1]{AM}, and hence the same holds for its factor group $\pi_1^t(X/S,\bar x)$.

\ms\noi
From now on we assume that  $S$ is integral, pure-dimensional (i.e., $\dim  S= \dim \mathcal{O}_{S,s}$ for every closed point $s\in S$), separated and excellent. Furthermore, we assume that $X\to S$ is separated and of finite type.
For integral $X$ we put
\[
\dim_S X:= \mathrm{deg.tr.}(k(X)/k(T)) + \dim_{\textrm{Krull}} T,
\]
where $T$ is the closure of the image of $X$ in $S$. If the image of $X$ in $S$ contains a closed point of $T$,  then $\dim_S X=\dim_{\textrm{Krull}} X$ by \cite[5.6.5]{EGAIV.4}.
We call $X$ an \emph{$S$-curve}, if $\dim_S X=1$. For a regular connected $S$-curve $C$, let $\bar C$ be the unique regular compactification of $C$ over $S$. Then the finite tame Galois covers of $C$ are exactly those \'{e}tale Galois covers which are tamely ramified along $\bar C \smallsetminus C$ in the classical sense (e.g.\ \cite{GM71}).
We recall the notion of curve-tameness from \cite{KeSch10}:

\begin{definition}
Let $Y\to X$ be an \'{e}tale cover of separated $S$-schemes of finite type. We say that $Y\to X$ is {\em curve-tame} if for any morphism $C\to X$ with $C$ a regular $S$-curve, the base change $Y\times_X  C \to  C$ is tamely ramified along $\bar C\smallsetminus C$.
\end{definition}

After choosing a geometric point $\bar x$, we obtain the \emph{curve-tame fundamental group} $\pi_1^{ct}(X,\bar x)$, which is the profinite quotient of $\pi_1^\et(X,\bar x)$ which classifies finite curve-tame Galois covers. This curve-tame fundamental group was considered in \cite{Sch-singhom}, \cite{KeSch09}, \cite{GS16}.

\begin{proposition}\label{pi1ct=pi1t}
The curve-tame fundamental group is the profinite completion of the tame fundamental group. If $X$ is geometrically unibranch, then both are isomorphic:
\[
\pi_1^{t}(X,\bar x) \cong \pi_1^{ct}(X,\bar x).
\]
\end{proposition}

For the proof we recall the ``Key Lemma'', Lemma 2.4 of \cite{KeSch10}. The normalization of an integral scheme $Z$ in its function field is denoted by $\widetilde Z$.

\begin{lemma} \label{key-lemma} Let $A$ be a local, normal and excellent ring and let
$X'\subset X=\Spec(A)$ be a nonempty open subscheme. Let $Y'\to X'$ be an \'{e}tale Galois cover of prime degree $p$.  Assume that $X\smallsetminus X'$ contains an irreducible component $D$ of codimension one in $X$ such that $Y'\to X'$ is ramified along the generic point of\/ $D$. Then there exists an integral, closed subscheme $C$ of $X$ of Krull-dimension one with $C':=C \cap X'\ne\varnothing$ such that the base change $Y'\times_{X'} \widetilde C' \to \widetilde C'$ is ramified along a point of\/ $\widetilde C \smallsetminus \widetilde C'$.
\end{lemma}

\begin{proof}[Proof of \cref{pi1ct=pi1t}]
We have to show that a finite \'{e}tale Galois cover is tame if and only if it is curve-tame. Both notions coincide for regular $S$-curves and are stable under base change. Hence finite tame Galois covers are curve-tame and it remains to show the converse.

So let $f: Y\to X$ be a finite \'{e}tale Galois cover, $y\in Y$ a point, $x=f(y)$ and $w\in \Val_S k(y)$ such that $w$ is wildly ramified in $k(y)/k(x)$. We have to find an $S$-curve $C$ and a morphism $\varphi: C\to X$ such that $C\times_XY\to C$ is wildly ramified at some point in $\bar C \smallsetminus C$.

We first assume that $\Gal(Y/X)$ is cyclic of order $p:=\ch k(w)$. Then, since $w$ is ramified, $k(y)/k(x)$ is also of degree $p$. Let $X'$ be the closure of $\{x\}$ in $X$ with reduced scheme structure and $Y'=X'\times_XY$. In order to find $\varphi: C\to X$, we may replace $X$ by $X'$ and $Y$ by the connected component of $y$ in $Y'$, i.e., we may assume that $x$ and $y$ are the generic points of the integral schemes $X$ and $Y$. Replacing $X$ by its normalization $\tilde X$ and $Y$ by its base change along $\tilde X\to X$, we may assume that $X$ and $Y$ are normal.

By quasi-purity of the branch locus \cite{quasi-purity}, we find a geometric discrete rank-one $S$-valuation $v$ on $k(y)$ with $\ch k(v)=p$ that is (wildly) ramified in $k(y)/k(x)$.
Then we can find a normal compactification (over $S$) $\bar X$ of $X$ such that the center $D$ of $v$ on $\bar X$ is of codimension one.  Applying  \cref{key-lemma} to the local ring of some closed point of $\bar X$ contained in $D$, we obtain a morphism $\varphi: C\to X$ with the required property.

\smallskip\noindent
It remains to reduce the general case to the cyclic-order-$p$-case. For this let
\[
A\subset R_w(k(y)/k(x)) \subset \Gal(k(y)/k(x))
\]
be a subgroup of order $p$ (here $R_w$ denotes the ramification group) and let $X'=Y_A$. Then $Y/X'$ is wildly ramified and we find a morphism $\varphi: C \to X'$ with the required property. Then the composite of $\varphi$ with the projection $X'\to X$ yields what we need.
\end{proof}

\section{Joins of hensel rings}

In this section we recall some results of M.~Artin \cite{Art71}. We need them in slightly higher generality, therefore we have to go into a little more detail.

\medskip\noindent
Let $A$ be a normal integral domain, and $\p$, $\q$ prime ideals in $A$. Suppose we are given embeddings of the henselizations $A_\p^h$, $A_\q^h$ into a separable closure  $\bar K$ of $K$. We call the ring $D=[A_{\p}^h,A_{\q}^h]\subset \bar K$ generated by $A_\p^h$ and $A_\q^h$ their \emph{join} (with respect to the chosen embeddings to $\bar K$):
\[
\begin{tikzcd}
  &\bar{K}\arrow[d,dash]\\
  &D=[A_{\p}^h,A_{\q}^h]\arrow[dl,dash]\arrow[dr,dash]\\
A_{\p}^h\arrow[dr,dash]&&A_{\q}^h\arrow[dl,dash]\\
&A
\end{tikzcd}
\]
Artin's main technical result is the following.
\begin{theorem}[{\cite[Theorems 2.2, 2.5]{Art71}}]
\label{artin-2.5} The join $D=[A_{\p}^h,A_{\q}^h]$ is a local henselian ring. If neither of the primes $\p$, $\q$ contains the other, then $D$ is strictly hen\-selian and the pullback of its maximal ideal to $A$ is strictly contained in $\p$ and~$\q$. \qed
\end{theorem}

The following lemma will prove useful.

\begin{lemma}\label{formallyreal}
Let $A\subset D$ be an extension of integral domains and assume that $D$ contains $A_\p^h$ and $A_\q^h$ for primes $\p,\q\subset A$ such that none of\/ $\p,\q$ contains the other. Then $-1$ is a sum of squares in $D$. In particular, no residue field of $D$ is formally real.
\end{lemma}

\begin{proof}
We follow the argument of \cite[Proof of Theorem~2.5]{Art71} with a small modification. Assume $\ch k(\p)=p>0$. Then $A_\p^h$ either contains $\F_p$ or the ring of algebraic $p$-adic integers. In both cases, $-1$ is a sum of squares in $A_\p^h$. It therefore  remains the case that the residue fields of both $\p$ and $\q$ have characteristic zero.

If $\p +\q =A$, then we can choose $a\in A$ with $a\equiv 1 \bmod \p$ and $a\equiv -1 \bmod \q$. Then $a$ is a unit and a square in $A_\p^h$ and $-a$ is a square in $A_\q^h$. Hence $-1=a^{-1}(-a)$ is a square in $D$.

Finally, assume that $\p+\q \neq A$. Choose $g\in \p \smallsetminus \q$, $h \in \q\smallsetminus \p$ and set $f=g+h$. Then $f\in \p+\q$ but $f\notin \p \cup \q$.  Replacing $A$ by $A[f^{-1}]\subset A_\p \cap A_\q\subset D$ does not change $A_\p^h$ and $A_\q^h$.
Since $\p$ and $\q$ are coprime in $A[f^{-1}]$, the argument of the last paragraph applies to show that $-1$ is a square in $D$.
\end{proof}

Let $A$ be a ring. The connected components of $\Spec(A)$ are of the form $\Spec(\bar A)$, where $\bar A$ is a colimit of rings of the form $A/eA$ with idempotent elements $e\in A$. We call the rings $\bar A$ occurring in this way the \emph{components} of $A$.

\medskip
Following M. Artin \cite{Art71}, we call a ring $A$ \emph{quasi-acyclic} if every component of $A$ is  a local henselian ring.
By \cite[Proposition 3.3]{Art71}, filtered colimits of quasi-acyclic rings and finite algebras over quasi-acyclic rings are quasi-acyclic.
In particular,  an integral algebra over a quasi-acyclic ring is quasi-acyclic. The main result of \cite{Art71} (though formulated as the slightly weaker statement Theorem~3.4 there) is the following:

\begin{theorem}[Artin]\label{artinslemma}
Let $A$ be a ring, $\p_1,\p_2\subset A$ prime ideals and, for $i=1,2$, $A_{\p_i}^h$ the henselization of the local ring $A_{\p_i}$. Let, for $i=1,2$,  $B_{i}$ be local, integral $A_{\p_i}^h$-algebras.
Then
\[
B= B_{1} \otimes_A B_{2}
\]
is quasi-acyclic. Let $D$ be a component of $B$. Then one of the following holds.
\begin{enumerate}
  \item If both ring homomorphism $B_i \to D$ are not local, then the residue field $D/\m_{D}$ of $D$ is separably closed.
  \item If, say, $B_1\to D$ is local, it is integral. In particular, $D/\m_D$ is an algebraic field extension of $B_1/\m_{B_1}$.
\end{enumerate}
\end{theorem}

\begin{proof}
Because $B$ is integral over $A_{\p_1}^h\otimes_A A_{\p_2}^h$, we may assume that  $B_i=A_{\p_i}^h$.

     We follow the arguments given in \cite{Art71}:
If $A$ is a factor ring of a ring $A'$ and the statement holds for $A'$, then it holds for $A$. Therefore we may assume that $A$ is a normal domain. Let $K$ be the quotient field of $A$, $\bar K/K$ a separable closure and $\bar{A}$ the integral closure of $A$ in $\bar{K}$. Then every component $D$ of $B$ is the limit of \'{e}tale $A$-algebras, hence a normal domain. Therefore, we can embed $D$ into $\bar{K}$. The induced maps $A_{\p_1}^h\to D$ and $A_{\p_2}^h\to D$  generate $D$ because $D$ is a quotient of $B=  A_{\p_1}^h\otimes_A A_{\p_2}^h$.
Hence we are in the situation of \cref{artin-2.5} and this theorem implies that $D$ is a local henselian ring. Now let $\bar{\p}_1, \bar{\p}_2 \subset \bar{A}$ be the unique extensions of $\p_1$ and $\p_2$ to $\bar{A}$, such that $\bar{\p}_i\bar{A}_{\p_i}\cap A_{\p_i}^h=\p_i A_{\p_i}^h$.
For a subextension $K'/K$ of $\bar{K}/K$ we let $A'$ be the normalization of $A$ in $\bar{K}$ and $\p_i'=\bar{\p}_i\cap A'$.

If, say,  $\bar{\p}_1 \subset\bar{\p}_2$, then $\p_1\subset \p_2$ and we are in case (2). Otherwise there is a finite field extension $K'/K$ such that~$\p_1'$ and~$\p_2'$ are not contained in one another.
Then $D' = [A'^h_{\p_1'},A'^h_{\p_2'}]$ is the integral closure of $D$ in the quotient field of $D'$. By \cref{artin-2.5}, $D'$ is henselian with separably closed residue field $D'/\m_{D'}$. Since the extensions of the residue fields is finite, we conclude that $D/\m_D$ is separably closed, as well.  Here, we have to exclude the possibility that $D/\m_D$ is real closed but not separably closed.  This cannot happen by \cref{formallyreal}.
\end{proof}

We call a ring homomorphism $f: A\to B$ \emph{ind-\'{e}tale} if it is a filtered colimit of \'{e}tale ring homomorphisms $f_i: A \to B_i$. If $A\to B$ is ind-\'{e}tale with $B$ local, henselian, then $B$ is an integral $A_\p^h$-algebra for some prime ideal $\p\subset A$. Hence the following \cref{Artins-thm} is an immediate consequence of \cref{artinslemma}.

\begin{theorem}[Artin]\label{Artins-thm}
Let $A$ be a ring and let $B_1,\ldots, B_n$ be ind-\'{e}tale quasi-acyclic $A$-algebras. Then their tensor product
\[
B= B_1\otimes_A \ldots \otimes_A B_n
\]
is quasi-acyclic. For a maximal ideal $\m\subset B$ one of the following holds
\begin{enumerate}
  \item $B/\m$ is separably closed, or
  \item $B/\m$ is a separable, algebraic extension of $B_i/\m_i$, where  $\m_i\subset B_i$ is a maximal ideal for some $i$, $1\leq i\le n$.
\end{enumerate}
\end{theorem}

\section{Comparison with \v{C}ech cohomology} \label{sec:cech}

In this section we assume that all rings and schemes
lie over a fixed quasi-compact and quasi-separated  base scheme~$S$.

\begin{definition}
A ring homomorphism $f: A\to B$ is called \emph{tame} if it is  \'{e}tale and for every prime ideal $\p\subset A$ and every $S$-valuation $v$ of $k(\p)$ there exists a prime ideal $\q\subset B$ over $\p$ and an $S$-valuation $w$ of $k(\q)$ over $v$ such that $w/v$ is tamely ramified (in particular, $f$ is faithfully flat).
\end{definition}

\begin{remark} A ring homomorphism $A\to B$ is tame if and only if the associated morphism $\Spec(B)\to \Spec(A)$ is a covering in $(\Spec(A)/S)_t$. The composite of tame morphisms is tame.  Tame ring homomorphisms are stable under base change. The tensor product of tame $A$-algebras is a tame $A$-algebra.
\end{remark}

\begin{definition} A ring $A$ is \emph{tamely henselian} if $\Spec(A)$ is connected and every tame ring homomorphism $A\to B$ splits.
\end{definition}

\begin{example}
The tame henselizations of \cref{tame-henselization-def}  are tamely henselian rings.
\end{example}

\begin{remark}\label{tame-henselian-remark}
$A$ is tamely henselian if and only if  $\Spec(A)$ is connected and every tame covering of $\Spec(A)$ splits. This follows since every tame covering of $\Spec(A)$ has a finite subcovering by \cref{finite-cov}.
\end{remark}

\begin{lemma} \label{tamely hens charact} The following are equivalent.
\begin{enumerate}[\rm (i)]
  \item $A$ is tamely henselian
  \item $A$ is local, henselian and there exists an $S$-valuation~$v$ of  the residue field $k=A/\m_A$  such that $k$ does not admit any nontrivial, finite separable field extension $k'/k$ which is tamely ramified at some extension $w$ of $v$ to $k'$.
\end{enumerate}
\end{lemma}

\begin{proof} (ii)$\Rightarrow$(i):  Let $A\to B$ be a tame homomorphism. Since $A$ is local henselian, we have a decomposition
\[
B= B_0 \times B_1 \times \cdots \times B_n,
\]
such that the maximal ideal $\m$ of $A$ is not in the image of  $\Spec B_0 \to \Spec A$, and $A\to B_i$ is finite \'{e}tale with $B_i$ local henselian for $i=1,\ldots,n$.  Since $A\to B$ is tame, we can find an index $i$ such that there is a valuation $w$ of $B_i/\m_i$ lying over $v$ such that $w/v$ is tamely ramified. By assumption $A/\m\to B_i/\m_i$ is an isomorphism, hence  $A\to B_i$ splits.

\noindent
(i)$\Rightarrow$(ii):
Let $A$ be tamely henselian. Assume that $A$ is not local henselian  and let $\m\subset A$ be a maximal ideal. Then (by the construction of the henselization) there exists an \'{e}tale $A$-Algebra $B$ such that there is exactly one maximal ideal $\n\subset B$ over $\m$, $A/\m\to B/\n$ is a field isomorphism and $A\to B$ does not split.
Then  $(\Spec A \smallsetminus \{\m\}) \cup \Spec B \to \Spec A$ is a tame covering of $\Spec (A)$ which does not split.
Hence $A$ is not tamely henselian by \cref{tame-henselian-remark}. We conclude that $A$ is local henselian.

Assume now that for every $S$-valuation $v$ of $k=A/\m$ there is a finite, separable field extension $k^v/k$ which is tamely ramified at some extension $w$ of $v$ to $k^v$. Then
\[
\coprod_{v\in \Val_S(k)} \Spec k^v \lang \Spec k
\]
is a tame covering.  By \cref{finite-cov}, there exists a finite subcovering. Hence we find a finite tame $k$-algebra $k\to R$ which does not split. Let $A\to B$ be the finite \'{e}tale morphism lifting $k\to R$. Then $(\Spec A \smallsetminus \{\m\}) \cup \Spec B \to \Spec A$ is a tame covering of $\Spec A$ which does not split. Contradiction.
\end{proof}

\begin{corollary}
A finite algebra over a tamely henselian ring is the finite product of tamely henselian rings.
\end{corollary}

\begin{proof}
A finite algebra over a local henselian ring is the finite product of local henselian rings. Moreover, the property of the residue field that it is closed under extensions which are tamely ramified with respect to a fixed valuation extends to finite field extensions.
\end{proof}

\begin{definition} A scheme $X$ is \emph{\'{e}tale} (resp.\ \emph{tamely})  \emph{acyclic} if every \'{e}tale resp.\ tame covering $X'\to X$ splits.
A ring $A$ is \'{e}tale resp. tamely acyclic if the scheme $\Spec A$ is \'{e}tale (resp.\ tamely) acyclic.
\end{definition}

\begin{proposition}\label{acyclic-comp} A ring
$A$ is \'{e}tale (resp.\ tamely) acyclic if and only if every component of $A$ is a strictly (resp.\ tamely) henselian ring.
\end{proposition}

\begin{proof}
The assertion in the \'{e}tale case is given in \cite[Proposition 3.2]{Art71}. The proof in the tame case is word by word the same.
\end{proof}

From  \cref{Artins-thm}, we obtain

\begin{theorem}\label{product of tamely acyclic}
Let $A$ be a ring and let $B_1,\ldots, B_n$ be ind-\'{e}tale $A$-algebras  which are \'{e}tale (resp.\ tamely) acyclic . Then their tensor product
\[
B= B_1\otimes_A \ldots \otimes_A B_n
\]
is \'{e}tale (resp.\ tamely) acyclic.
\end{theorem}

\begin{proof}
By \cref{Artins-thm}, $B$ is quasi-acyclic. It remains to show that every component is strictly (resp.\ tamely) henselian. We already know that the components are local henselian. Hence it suffices to show that the residue fields are separably closed  (resp.\ tamely closed with respect to some $S$-valuation). This is clear in case (1) of \cref{Artins-thm}, where the residue field is separably closed and follows in case (2) from the assumption on the $B_i$.
\end{proof}

\begin{definition}
We call a ring homomorphism $f: A\to B$ \emph{ind-tame} if it is a filtered colimit of tame homomorphisms $f_i: A \to B_i$.
\end{definition}
Since an inverse limit of finite nonempty sets is nonempty, $f$ is ind-tame if and only if it is  ind-\'{e}tale and for every prime ideal $\p\subset A$ and every $S$-valuation $v$ of $k(\p)$ there exists a prime ideal $\q\subset B$ over $\p$ and an $S$-valuation $w$ of $k(\q)$ over $v$ such that $w/v$ is tamely ramified. In particular, $f$ is faithfully flat.

\begin{proposition}\label{acyclic tame}
For every ring $A$ there exists a faithfully flat ind-\'{e}tale (resp.\ ind-tame)  $A$-algebra $ \widetilde A$ which is \'{e}tale (resp.\ tamely) acyclic.
\end{proposition}

\begin{proof}
We restrict to the tame case and follow the method of Bhatt-Scholze, \cite[Proof of Lemma 2.2.7]{BS15}. Let $I$ be the set of isomorphism classes of tame $A$-algebras.
 For each $i \in I$ pick a representative $A\to B_i$ and set $A_1$  to be their tensor product, i.e.,
 \[
  A_1:= \colim_{J \subset I~\text{finite}} \bigotimes\limits_{j \in J} B_j,
 \]
where the tensor product is taken over $A$ and the (filtered) colimit is indexed by the poset of finite subsets of $I$. There is an obvious ind-tame map $A\to A_1$, and it is clear
from the construction that any tame $A$-algebra $B$ admits a map to $A_1$,
i.e., the map $A \to B$ splits after base change to $A_1$. Iterating the construction with $A_1$ replacing $A$ and proceeding inductively defines a tower $A\to A_1 \to A_2 \to \cdots $ of  $A$-algebras with ind-tame transition maps. Set $\widetilde A = \colim A_n$. As tame morphisms of rings are finitely presented, one checks that any tame $\widetilde A$-algebra has a section, so $\widetilde A$ is tamely acyclic.
\end{proof}

We will make use of the following topological observation:

\begin{proposition}\label{topology}
Let $X$ be a spectral space such that every connected component of $X$ contains a unique closed point.  Then the composition
\[
X^c\hookrightarrow X\twoheadrightarrow \pi_0(X)
\]
from the subspace of closed points of $X$ to the quotient space of connected components is a homeomorphism of profinite sets.
\end{proposition}

\begin{proof}  By \cite[TAG 0906]{stacks-project}, $\pi_0(X)$ is profinite. By assumption, $X^c \to \pi_0(X)$ is bijective (and continuous). We conclude that $X^c$ is Hausdorff. Moreover, being homeomorphic to $\Specm (R)$ for some commutative ring $R$, $X^c$ is quasi-compact. Therefore $X^c$ is compact and the continuous bijection is a homeomorphism.
\end{proof}

By an affine, pro-\'{e}tale $X$-scheme, we mean an inverse limit of affine \'{e}tale $X$-schemes. This is more restrictive than `pro-\'{e}tale' in the sense of \cite{BS15}.

\begin{proposition} \label{fibre product acylic} Let $X$ be a  scheme having the property that every finite subset of\/ $X$ is contained in an affine open.
 Suppose we are given affine pro-\'{e}tale $X$-schemes $U_i=\Spec A_i$, $i=1,\ldots,n$, and let
\[
U=U_1 \times_X \cdots \times_X U_n
\]
be their fibre product.
\begin{enumerate}[\rm (i)]
  \item If all $A_i$ are \'{e}tale acyclic, then $U$ is affine, $U=\Spec A$ with $A$ \'{e}tale acyclic.
  \item Let $X$ be an $S$-scheme. If all $A_i$ are tamely acyclic, then $U$ is affine, $U=\Spec A$ with $A$ tamely acyclic.
\end{enumerate}
\end{proposition}

\begin{remark} \label{separated-rem} A  scheme having the property that every finite subset is contained in an affine open is separated (\cite[TAG 01KP]{stacks-project}).
\end{remark}

\begin{proof}[Proof of \cref{fibre product acylic}]
If there is an affine Zariski-open subscheme $V\subset X$ such that all $U_i$ map to~$V$, the result follows from \cref{product of tamely acyclic}.

In the general case, let $p_i\in U_i$, $i=1,\ldots,n$, be closed points. By assumption, there is an affine, Zariski-open subscheme $V=\Spec B \subset X$ containing the images of $p_1,\ldots,p_n$. Since $U_i$ is quasi-acyclic, every connected component of $U_i$ has a unique closed point and it maps to $V$ if and only if its closed point maps to $V$. By \cref{topology}, $\varphi_i: \Specm A_i \to  \pi_0(U_i)$ is a homeomorphism of profinite spaces. Therefore, we find a closed and open subset $W_i \subset \Specm A_i$ which contains $p_i$ and maps to $V$. Then also the preimage $V_i\subset U_i$ of $\varphi_i(W_i)\subset \pi_0(U_i)$ under $U_i\to \pi_0(U_i)$ maps to $V$. Hence $V_i$ is a closed and open subset of $U_i$ containing $p_i$ and mapping to $V$. Being closed and open in $U_i$, $V_i$ is the spectrum of an ind-\'{e}tale, \'{e}tale (resp.\ tamely) acyclic $B$-algebra. By the first part of the proof, $V(p_1,\ldots,p_n):=V_1\times_X\cdots \times_X V_n$ is a closed and open subscheme of $U$, affine, ind-\'{e}tale and \'{e}tale (resp.\ tamely) acyclic. Varying the points $p_1,\ldots, p_n$, the $V(p_1,\ldots,p_n)$ cover $U$. Since $U$ is quasi-compact, we find a finite subcovering. Replacing the $V(p_1,\ldots,p_n)$ by closed and open subschemes, which are ind-\'{e}tale and \'{e}tale (resp.\ tamely) acyclic, we may assume that the finite union is disjoint. Hence $U$ is affine, ind-\'{e}tale and \'{e}tale (resp.\ tamely) acyclic.
\end{proof}

Next we generalize Artin's theorem on the comparison between \'{e}tale \v{C}ech and sheaf cohomology to the tame case. We also obtain a slightly sharper statement for the \'{e}tale case by weakening the Noetherian assumption to quasi-compactness.

\begin{theorem}[Comparison with \v{C}ech cohomology] \label{cechcompare} Let $X$ be a quasi-compact scheme having the property that every finite subset of $X$ is contained in an affine open, $S$ a quasi-compact and quasi-separated scheme and $X\to S$ a morphism. Then, for every presheaf $P$ of abelian groups on $X_\et$ with \'{e}tale resp.\ tame sheafification $P^{\# \et}$ resp.\ $P^{\#t}$, the natural maps
\[
\check{H}^n_\et (X,P) \to H^n_\et(X, P^{\#\et}) \text{ and }\check{H}^n_t(X/S,P) \to H^n_t(X/S, P^{\# t})
\]
are isomorphisms for all $n\ge 0$.
\end{theorem}

\begin{proof} We give the argument in the tame case, the \'{e}tale case is analogous.
The statement of the theorem is a formal consequence of the assertion
\[
\check{H}^n_t(X/S,P)=0 \text{ for all }n\ge 0 \text{ and every presheaf }P\text{ with }P^{\# t}=0. \leqno (\dagger)
\]
Indeed, if $(\dagger)$ holds, then $\check{H}^\bullet_t(X/S,-)$ is a $\delta$-functor on \emph{sheaves}. Since injective sheaves are injective as presheaves, the functor is also erasing. This shows $\check{H}^n_t (X/S,F) \cong H^n_t(X/S, F)$ for every sheaf $F$. It remains to see that $\check{H}^n_t (X/S,P)\to \check{H}^n_t (X/S,P^{\#t})$ is an isomorphism for every presheaf $P$. Splitting the presheaf homomorphism $P \to P^{\# t}$ into
\[
0\to K \to P \to G\to 0, \quad 0\to G \to P^{\# t} \to C \to 0,
\]
we have $K^{\# t}=0=C^{\# t}$ and the statement follows from the long exact sequence for \v{C}ech cohomology of presheaves.

To prove $(\dagger)$, we cover $X$ by finitely many affine Zariski-open schemes.
By \cref{acyclic tame}, and forming the disjoint union,  we find an affine, tamely acyclic, pro-tame covering $U \to X$, i.e.,
 $U=\lim U_i$ is a limit of affine, tame coverings $U_i\to X$. Since every tame covering of $U$ splits,  the system $(U_i)$ is cofinal in the system of all tame coverings of $X$ (partially ordered by refinement).  For the \v{C}ech cohomology, we obtain
 \[
  \check{H}^p_t(X/S, P)= \colim_i \check{H}^p(U_i \to X,P)= H^p( \colim_i \check{C}^\bullet(U_i \to X,P)).
 \]
 For an $X$-scheme $Y$, we denote by  $Y^n$ the $n$-th self fibre product over $X$. Then, for  $n\ge 1$, $U^n=\lim_i U_i^n$ is tamely acyclic by \cref{fibre product acylic}.  Since $P$ has trivial sheafification, we obtain
 $\colim_i P(U_i^n)=0$. Hence the colimit of the \v{C}ech complexes $\check{C}^\bullet(U_i \to X,P)$ vanishes. In particular, the colimit has trivial cohomology.
\end{proof}

\section{Comparison with \'{e}tale cohomology}
In this section we compare tame and \'{e}tale cohomology in two situations. First, when the coefficients are invertible on $S$ and second, if $X\to S$ is proper.

\medskip
Since the \'{e}tale site is finer than the tame site and has the same underlying category, any \'{e}tale sheaf can be considered as a tame sheaf. To put this more formal, we denote the natural morphism of sites by $\alpha: X_\et\to (X/S)_t$. Then $\alpha_* F$ is the same presheaf as $F$ but considered as a tame sheaf.
\begin{proposition}[Invertible coefficients] \label{compare invertible} Let $m\ge 1$ be an integer invertible on $S$ and let $F$ be an \'{e}tale sheaf of $\Z/m\Z$-modules on $X$. Then
\[
H^n_t(X/S, \alpha_* F)\cong H^n_\et(X,F)
\]
for all $n\ge 0$.
\end{proposition}

\begin{proof} Using the Leray-Serre spectral sequence, it suffices to show that the higher direct image sheaves $R^q\alpha_*F$ vanish for $q\ge 1$.  This can be checked on stalks at the conservative family of tame points of described in \cref{enoughpoints}.

Let $x\in X$ be a (scheme-)point, $k(x)^s$ a separable closure of $k(x)$, $v$ an $S$-valuation on $k(x)$, $\bar v$ a prolongation to $k(x)^s$, $K_v:=(k(x)^s)^{R_{\bar v}}$ the fixed field of its ramification group and $L$ a finite subextensions of $k(x)^s/K_v$. Let $\bar x_L: \Spec(L)  \to X$ be the induced morphism and $(\bar x_L, \bar v|_L)$ the associated tame point of $X$. Since, by \cite[VI, 5.8]{SGA4},  \'{e}tale cohomology commutes with inverse limits of quasi-compact and quasi-separated schemes with affine transition maps,  we have an isomorphism
\[
(R^q\alpha_*F)_{(\bar x_L, \bar v|_L)}\cong H^q_\et(\Spec \O_{(\bar x_L, \bar v|_L)}, F) ,
\]
where $\O_{(\bar x_L, \bar v|_L)}$ is the stalk of the structure sheaf at $(\bar x_L, \bar v|_L)$ described in \cref{stalkremark}. This ring is henselian with residue field $L$, hence
\[
(R^q\alpha_*F)_{(\bar x_L, \bar v|_L)} \cong H^q_\et(\Spec L, (\bar x_L)^*F).
\]
Let $p$ be the residue characteristics of $v$. Then the absolute Galois group of $L$ is a pro-$p$-group and since $(p,m)=1$, $H^q_\et(\Spec L, (\bar x_L)^*F)$ vanishes for $q\ge 1$ by \cite[(1.6.2)]{NSW}.
\end{proof}

\begin{proposition} [Proper morphisms] \label{proper-compare} Let $X$ be a quasi-compact scheme having the property that every finite subset of $X$ is contained in an affine open, $S$ a quasi-compact and quasi-separated scheme and $X\to S$ a proper morphism. Then, for every sheaf $F$ of abelian groups on $(X/S)_t$ with \'{e}tale sheafification $F^{\#\et}$ the natural map
\[
H^n_t(X/S, F)\to H^n_\et(X,F^{\#\et})
\]
is an isomorphism for all $n\ge 0$.
\end{proposition}

\begin{proof} By \cref{cechcompare}, the horizontal arrows in the commutative diagram
\[
\begin{tikzcd}
  \check{H}^n_t(X/S, F)\rar\dar  & H^n_t(X/S, F)\dar\\
  \check{H}^n_\et(X,F)\rar& H^n_\et(X,F^{\# \et})
\end{tikzcd}
\]
are isomorphisms. Since $X\to S$ is proper, every \'{e}tale covering of $X$ is a tame covering. Hence the left vertical arrow is an isomorphism and therefore also the right one.
\end{proof}

\section{Finiteness in dimension 1}\label{sec:finite}

We consider the arithmetic case.

\begin{lemma}\label{arithm-local}
Let $K/\Q$ be a number field, $\p\subset \O_K$ a prime ideal, $\O_{K,\p}^h$ the henselization of $\O_{K,\p}$ and $K_\p^h$  its quotient field.

Let $X=\Spec \O_{K,\p}^h$, $x \in X$ the closed point, $\eta\in X$ the generic point, and put $S=\Spec \Z$. Then for every tame torsion sheaf $F$ on $X/S$, we have $H^q_t(X/S,F)=0$ for $q>1$, $H^q_t(\eta/S, F)=0$ for $q>2$ and $H^q_{t, \{x\}}(X/S,F)=0$ for $q>3$.
If $F$ has finite stalks, these groups are finite for all $q$.
\end{lemma}

\begin{proof}
The only valuation on the finite field $k(x)=\O_K/\p$ is the trivial one. Since $K_\p^h$ is an algebraic extension of $\Q$, every valuation on $k(\eta)=K_\p^h$ is trivial or of rank one. First, we have the valuation $v_\p$ associated to the valuation ring $\O_{K,\p}^h$.  Let $K^t_\p/K_{\p}^h$ be the maximal tamely ramified extension (with respect to $v_\p$) of  the henselian local field $K_{\p}^h$ in an algebraic closure $\bar K$.

Then, by \cite[7.2.5]{NSW}, $G=\Gal(K_\p^t/K_{\p}^h)$ is the semidirect product of $\hat \Z$ with $\hat \Z^{(p')}= \prod_{\ell \neq p}\hat \Z_\ell$, where $p$ is the prime number below~$\p$. In particular, $G$ has cohomological dimension~$2$ and the cohomology of $G$ with values in finite discrete modules is finite.

Let $v$ be another valuation on $K_\p^h$. If $v$ is trivial, every extension of $K_\p^h$ is tame with respect to $v$. If $v$ is nontrivial, then, since different decomposition groups in the absolute Galois group of a number field have trivial intersection by \cite[12.1.3]{NSW}, $v$ is completely decomposed in $\bar K/K_{\p}^h$. We conclude that $\Spec K_\p^t/\Spec K_\p^h$ is a profinite tame cover of $K_\p^h$  with respect to $S=\Spec \Z$. Moreover, every tame cover of $\Spec K_\p^t$ splits. Hence for every sheaf $F\in \Sh_t(\Spec K_\p^t/\Spec \Z)$ the Hochschild-Serre spectral sequence induces isomorphisms
\[
H^q_t(\Spec K_\p^h/S, F) \cong H^q(G, \Gamma (\Spec K_\p^t, F)).
\]
This shows that these groups are trivial for $q>2$ if $F$ is a torsion sheaf and finite for all $q$ if $\Gamma (\Spec K_\p^t, F)$, which is the stalk of $F$ at $(\eta,v_\p)$, is finite.

A similar argument with $\Gal(\O_{K,\p}^\sh / \O_{K,\p}^h )\cong \hat \Z$ shows that $H^q_t(\Spec \O_{K,\p}^h/S, F)$ is trivial for $q>1$ if $F$ is a torsion sheaf, and finite for all $q$ if $F$ has finite stalks. Finally, the statement on the  cohomology groups with support follows from the result for the cohomology of $\Spec \O_{K,\p}^h$ and $\Spec K_\p^h$, and the long exact sequence connecting them.
\end{proof}

\begin{theorem}
Let $X \to S=\Spec \Z$ be quasi-finite and let $F\in \Sh_t(X/S)$ be a torsion sheaf. Assume that $F$ has finite stalks and its \'{e}tale sheafification $F^{\# \et}$ is constructible. Then the tame cohomology groups
\[
H^q_t(X/S,F)
\]
are finite for all $q$. For an arbitrary torsion sheaf $F\in \Sh_t(X/S)$ the following holds.
\begin{enumerate}[\rm (i)]
  \item If $\dim X=0$, then $H^q_t(X/S,F)=0$ for $q>1$.
  \item If $\dim X=1$ assume that $X(\R)=\varnothing$ or that $F$ has trivial $2$-torsion. Then $H^q_t(X/S,F)=0$ for $q >3$.
\end{enumerate}
\end{theorem}

\begin{proof} Using Mayer-Vietoris and induction on the cardinality of a  covering by separated open subschemes, we may assume that $X$ is separated. By the topological invariance of tame cohomology (\cref{topinvariance}), we may assume that $X$ is reduced. If $\dim X=0$, $X$ is finite over $\Spec \Z$. If $\dim X=1$, we can find a scheme $\bar X$ which is finite over $\Spec \Z$, contains $X$ as a dense open subscheme and  is regular at all points $x \in \bar X \smallsetminus X$. We denote the open immersion by $j: X \hookrightarrow \bar X$. Then excision for the sheaf $j_! F$
on $\bar X$ yields the long exact sequence
\[
\cdots \to \bigoplus_{x \in \bar X \smallsetminus X} H^q_{t,\{x\}}(\bar X_x^h/S, j_! F) \to H^q_t(\bar X/S, j_! F) \to H^q_t(X/S,F) \to \cdots
\]
In view of \cref{arithm-local}, we are reduced to showing the result for $X=\bar X$.

For finite $X$, the assertion of the theorem follows by \cref{proper-compare} from the corresponding well-known result for \'{e}tale cohomology \cite{Mazur-etale}.
\end{proof}

The same arguments also show
\begin{theorem}
Let $k$ be a separably closed (resp.\ finite) field and $X \to S=\Spec k$ a separated scheme of finite type of dimension $\le 1$.  Then for every torsion sheaf $F\in \Sh_t(X/S)$ the tame cohomology groups
\[
H^q_t(X/S,F)
\]
are zero for $q >2$ (resp.\ q>3). If $F$ has finite stalks and its \'{e}tale sheafification $F^{\# \et}$ is constructible, these groups are finite for all~$q$.
\end{theorem}

\section{Review of Huber pairs}
We refer to \cite{Hu96} for basic notions on adic spaces and to \cite{HueAd} for the definition of the tame site on such spaces.

\bigskip
In this paper, by a \emph{Huber pair} we mean a pair of rings $(A,A^+)$ with $A^+ \subset A$ an integrally closed subring. We equip $A$ and $A^+$ with the discrete topology. We call an adic space $\cX$ \emph{discretely ringed}, if it is locally of the form $\Spa(A,A^+)$ with $(A,A^+)$ a Huber pair.

Points of $\Spa(A,A^+)$ are by definition (multiplicatively written) valuations $v: A \to \Gamma_v \cup \{ 0\}$  such that $v(a_+)\leq 1$ for all $a_+\in A^+$. We will frequently consider $v$ also as a valuation on the quotient field $Q(A/{\supp v})$, where $\supp v = \{a\in A \mid v(a)=0\}$ is the \emph{support} of $v$.
The topology on $\Spa(A,A^+)$ is generated by the sets
\[
\{ v \in \Spa(A,A^+) \mid v(a) \le v(b)\ne 0 \} \quad (a,b \in A).
\]
Recall that a subgroup $H$ of the ordered group $\Gamma_v$ is \emph{convex} if each $\gamma$ with $1\le \gamma\le \delta\in H$ already belongs to $H$.
A convex subgroup $H\subset \Gamma_v$ \emph{bounds} $A$ if for all $a\in A$ there exists an $h\in H$ with $v(a)\le h$. In this case, the assignment
\[
v|H: A \to H \cup \{0\}, \ a\mapsto \left\{\begin{array}{cl}
v(a)& \text{if }v(a)\in H\\
0 & \text{if } v(a)\notin H
\end{array}\right.
\]
defines a valuation in $\Spa(A,A^+)$ which is a specialization of $v$. We have $\supp v \subset  \supp v|H$ and $v|H$ is called a \emph{horizontal} (or primary) \emph{specialization} of $v$. The set of horizontal specializations of $v$ is totally ordered and has a minimal element~$\bar{v}$ corresponding to the ``characteristic subgroup'' $c\Gamma_v$ of $\Gamma_v$, the convex subgroup generated by the subset $\{v(a) \mid a\in A,\, v(a)\ge 1\}$.

\begin{definition} \label{RZ-points}
$v\in \Spa(A,A^+)$ is a \emph{Riemann-Zariski point} if it has no nontrivial horizontal specializations.
\end{definition}

\begin{remark} The terminology ``Riemann-Zariski'' is motivated by the fact that these points are in bijection to the points of the relative Riemann-Zariski space of $\Spec A^+$ with respect to $\Spec A$, cf.\ \cite[Corollary~3.4.7]{Tem11}.
\end{remark}

The minimal horizontal specialization $\bar v$ of $v$ is the unique Riemann-Zariski point among the horizontal specializations of $v$. Assume that $\bar{v}=\tr_\p$ is the trivial valuation on its support $\p=\supp \bar v$. Then for any $\q \in \Spec(A)$, $\p \subset \q$, the trivial valuation $\tr_\q$ with support $\q$ is a specialization of $\bar v$, hence of  $v$. A \emph{generalized horizontal specialization} is a horizontal specialization or a specialization  $v \rightsquigarrow \tr_\q$ as above.

\medskip
For an arbitrary convex subgroup $H\subset \Gamma_v$ the assignment
\[
v/H: A \to (\Gamma/H) \cup \{0\}, \ a\mapsto \left\{\begin{array}{ll}
v(a) \bmod H & \text{if }v(a)\ne 0\\
\; 0 & \text{if } v(a)=0
\end{array}\right.
\]
defines a valuation with $\supp (v/H)=\supp v$ and $v/H$ is a generalization of $v$. A specialization of type $v/H \rightsquigarrow v$ is called a \emph{vertical} (or secondary) \emph{specialization}.

\begin{lemma} \label{special-aff-chara} Let $(A,A^+)$ be a Huber pair and let $v,w \in \Spa(A,A^+)$.
\begin{enumerate}[\rm (i)]
  \item
$w$ is a vertical specialization of $v$ if and only if $\supp v = \supp w$ and  $\O_w \subset \O_v \subset Q(A/\supp v)$. In this case $\O_v$ is the localization of $\O_w$ at some prime ideal.
  \item $w$ is a horizontal specialization of $v$ if and only if there is a prime ideal $\p \subset \O_v$ such that \medskip

   \begin{enumerate}
     \item $A/\supp v \subset (\O_v)_\p$ (as subrings in $Q(A/\supp v)$),
     \item $\supp w$ is the preimage of $\p$ under $A\twoheadrightarrow A/\supp v \hookrightarrow (\O_v)_\p$, and
     \item $\O_w= (\O_v/\p)\cap Q(A/\supp w)$ (intersection in $Q(\O_v/\p)$).
   \end{enumerate}
 \[
  \begin{tikzcd}
  \O_v \arrow[r,two heads]\arrow[d,hook]&\O_v/\p\arrow[d,hook] &\\
  (\O_v)_\p \arrow[r,two heads]\arrow[d,hook]& Q(\O_v/\p)&\O_w\arrow[d,hook]\arrow[lu,hook']\\
  Q(A/\supp v)&A/\supp v \arrow[l,hook']\arrow[ul,hook'] \arrow[r]& Q(A/\supp w) \arrow[ul, hook'].
    \end{tikzcd}
 \]
\end{enumerate}

\end{lemma}
\begin{proof}
(i) is obvious by valuation theory.\\
(ii) Let $w$ be a horizontal specialization of $v$ and let $c\Gamma_v\subset H \subset \Gamma_v$ be a convex subgroup as in the definition.
Let $\p \subset Q(A/\supp v)$ be the set of all $x$ with $v(x) \notin H$. Since $H\supset c\Gamma_v$, we have $\p\subset \O_v$.
\begin{enumerate}
  \item $a\in \O_v$, $b\in \p$ $\Rightarrow$ $ab\in \p$: Otherwise, $v(ab) \in H$ and  $v(ab) \le v(b) \le 1$ would imply $v(b)\in H$, a contradiction.
  \item $a,b\in \p$ $\Rightarrow$ $a+b\in \p$: $v(a+b) \le \max(v(a),v(b)) \le 1$: By convexity, $v(a+b)\in H$ implies $(v(a)\in H)\vee (v(b)\in H)$.
  \item $a,b\notin \p$ $\Rightarrow$ $ab\notin \p$: clear since $v(ab)=v(a)v(b)$ and $H$ is a group.
\end{enumerate}
We conclude that $\p\subset \O_v$ is a prime ideal. For $a \in A/\supp v$ with $v(a) \le 1$ we have $a \in \O_v$ by definition. If $v(a)>1$, then $v(a) \in H$, hence $v(a^{-1}) \in H$ and $v(a^{-1})\le 1$. We conclude $a^{-1}\in \O_v \smallsetminus \p$, hence $a \in (\O_v)_\p$. This shows (a).
Moreover, $\supp w =\{a\in A \mid v(a) \notin H\}$ is the preimage of $\p$ under $A\to A/\supp v \subset Q(A/\supp v)$, showing (b). Finally, $\O_v/\p$ is the valuation ring of the valuation $v'$ on $Q(\O_v/\p)$ given on $\O_v$ by $v'(a)=0$ for $v(a)\notin H$ and $v(a)$ otherwise. We conclude that $w$ is the restriction of $v'$ to $Q(A/\supp w)$, hence $\O_w= \O_v/\p \cap Q(A/\supp w)$.

Conversely, assume there exists a prime ideal $\p\subset \O_v$ such that (a)--(c) hold. As is known from valuation theory, there is a convex subgroup $H\subset \Gamma_v$ such that $\p= \{x\in \O_v \mid v(x) \notin H\}$.  By (a), for all $x\in A$, there is a $b\in \O_v\smallsetminus \p$ with $xb \in \O_v$.
We conclude that $v(x)\le v(b^{-1}) \in H$, hence $H$ bounds $A$. By (b) we see that $\supp w = \supp v|H$, and then $c$ implies $w=v|H$.
\end{proof}

\begin{corollary}\label{vertical-RZ}
Every vertical generalization of a Riemann-Zariski point is a Rie\-mann-Zariski point.
\end{corollary}

\begin{proof}
If $v'$ is a vertical generalization of $v$, then by \cref{special-aff-chara} (i), $\supp v=\supp v'$ and $\O_{v'}=(\O_v)_\q$ for some prime ideal $\q \subset \O_v$. Assume that $v'$ has a non-trivial horizonal specialization. Then there exist a prime $(0)\subsetneq \p \subset \O_v'$ as in \cref{special-aff-chara} (ii). But $(\O_{v'})_\p=(\O_v)_\mathfrak{P}$  for some prime $\mathfrak{P} \subset \O_v$ with $(0)\subsetneq \mathfrak{P} \subset \q$. The diagram in \cref{special-aff-chara} (ii) shows that there is a non-trivial horizontal specialization of $v$, as well.
\end{proof}

We will use the following fact below:

\begin{proposition}[{\cite[Proposition 1.2.4]{HK94}}] \label{specialization-prop}
Any specialization in $\Spa(A,A^+)$ is a vertical specialization of a generalized horizontal specialization. \qed
\end{proposition}

A \emph{homomorphism of Huber pairs} $f: (A,A^+)\to (B,B^+)$ is a ring homomorphism $f: A\to B$ with $f(A^+)\subset B^+$. It is called \emph{integral}, if $A\to B$ and $A^+\to B^+$ are integral. It is called \emph{finite}, if it is integral and $A\to B$ is finite.
A homomorphism of Huber pairs  is called \emph{\'{e}tale}, \emph{tame} or \emph{strongly \'{e}tale} if the associated morphism of adic spaces $\Spa(B,B^+)\to \Spa(A,A^+)$ has this property. Explicitly, this means the following:

The homomorphism $f: (A,A^+)\to (B,B^+)$ is \emph{\'{e}tale} if $A\to B$ is an \'{e}tale ring homomorphism and $B^+$ is the integral closure in $B$ of a finitely generated $A^+$-subalgebra.

The homomorphism $f$ is \emph{tame}, resp.\ \emph{strongly \'{e}tale} if it is \'{e}tale and for every point $w$ of $\Spa(B,B^+)$ with image $v\in \Spa(A,A^+)$  the extension of valuations $w/v$ is tamely ramified resp.\ unramified in the separable algebraic field extension $k(\supp w)/k(\supp v)$.

\begin{remark}
In contrast to the scheme case, we do not assume that $\Spa(B,B^+)\to \Spa(A,A^+)$ is a tame \emph{covering}. Instead, we require tameness at every point of $\Spa(B,B^+)$.
\end{remark}

\medskip
Let $A$ be a local ring and~$A^+\subset A$ the preimage of a valuation ring $\O$ of the residue field~$k$ of~$A$. Then $A^+$ is local and, since valuation rings are integrally closed, $A^+$ is integrally closed in $A$. Hence $(A,A^+)$ is a Huber pair. We call such Huber pairs \emph{local}. If $(A,A^+)$ is a local Huber pair, then $\m_A$ is a prime ideal of $A^+$ and the natural map $A^+_{\m_A}=(A^+\smallsetminus \m_A)^{-1} A^+ \to A$ is an isomorphism. A homomorphism $f: (A,A^+)\to (B,B^+)$ of local Huber pairs is \emph{local} if $A^+\to B^+$ and $A\to B$ are local ring homomorphisms.

\medskip
For a Huber pair $(A,A^+)$ and $ v \in \Spa(A,A^+)$ we denote by $(A,A^+)_v$ the \emph{localization} of $(A,A^+)$ at~$v$. It is obtained as follows.
Let~$B:=A_{\supp v}$ be the localization of~$A$ at $\supp v$.
In its residue field~$k$ we have the valuation ring~$\O$ corresponding to~$v$. Let $B^+$ be the pre-image of $\O$ in $B$. Then $(A,A^+)_v=(B, B^+)$.

The adic space $\Spa((A,A^+)_v)$ is the intersection of all open neighborhoods of $v$ in $\Spa(A,A^+)$.
A point $w\in \Spa(A,A^+)$ specializes to $v$ if and only if $(A,A^+) \to (A,A^+)_w$ factors through $(A,A^+) \to (A,A^+)_v$.

\medskip
We start  the discussion of henselizations with the following observation.

\begin{lemma}\label{triple-henselian}
Let $A$ be a local ring with residue field $k$, $\pr: A\to k$ the projection and $\O\subset k$ a local subring. Set
\[
A^+=\pr^{-1}(\O) \subset A.
\]
Then $A^+$ is henselian if and only if $A$ and $\O$ are henselian.
\end{lemma}

\begin{proof} First observe that  $\m_A$ is a prime ideal of $A^+$.
If $\O$ and $A$ are henselian, then applying Hensel's lemma twice shows that $A^+$ is henselian:
Let $f$ be a monic polynomial over $A^+$ and $f=gh$ modulo $\m_{A^+}$ a decomposition with $g \bmod \m_{A^+}$ and $h \bmod \m_{A^+}$ monic and coprime.  We have $A^+/\m_{A^+}=\O/\m_\O$ and we can lift the decomposition to $\O$. Since $A$ is henselian with residue field $k\supset \O$, we can lift the decomposition to $A$ and the coefficients lie modulo $\m_A$ in $\O$. Hence they are in ~$A^+$.

Now assume that $A^+$ is henselian. Then, being a quotient ring of $A^{+}$, also $\O$ is henselian.  Assume that $A\to A'$ is finite with $A'$ connected. Then $A'$ is semi-local and, in order to show that $A$ is henselian, we have to show that $A'$ is local. For this it suffices to show that $A'/\m_A A'$ is local. Since this ring is a finite algebra over the field $A/\m_A$, it suffices to show that $A'/\m_A A'$ is connected.  Let $A'^+$ be the integral closure of $A^+$ in $A'$. Since all idempotents of $A'$ lie in $A'^+$ and $A'$ is connected, $A'^+$ is connected, as well.  Hence  $A'^+$ is local henselian and the same holds for  $A'^+/\m_A A'^+$.

\medskip\noindent
\emph{Claim:} The image of $\phi: A'^+/\m_A A'^+\to A'/\m_A A'$ ist integrally closed in  $A'/\m_A A'$.

\medskip\noindent
Using the claim we see that  $\mathrm{im}(\phi)$ has the same connected components as $A'/\m_A A'$. But $\mathrm{im}(\phi)$ is (henselian) local, hence connected, showing that $A'/\m_A A'$ is connected.

\medskip\noindent
It remains to show the claim. First we observe that $\m_AA'\subset A'^+$. Now assume that for $x\in A'$ there exists a monic polynomial $f\in A'^+[T]$ with $f(x)=a'\in \m_A A' \subset A'^+$. Then $x$ is a zero of  $f-a'\in A'^+[T]$, hence $x\in A'^+$.
\end{proof}

\begin{definition} Let $(A,A^+)$ be a local Huber pair, $k=A/\m_A$ and $\O$ the image of $A^+$ in $k$.
\begin{enumerate}
  \item $(A,A^+)$ is called \emph{henselian} if  $A^+$ is a henselian ring. By \cref{triple-henselian}, this is equivalent to $A$ and $\O$ being henselian.
  \item $(A,A^+)$ is called \emph{strongly henselian} if $A^+$ is a strictly henselian ring. By \cref{triple-henselian}, this is equivalent to $A$ being henselian and $\O$ being strictly henselian.
  \item $(A,A^+)$ is called \emph{tamely henselian} if $(A,A^+)$ is henselian and $k$ does not admit any nontrivial, finite separable field extension which is tamely ramified with respect to the (henselian) valuation corresponding to $\O$. Equivalently: $(A,A^+)$ is strongly henselian and $k$ is separably closed if the residue field $\kappa$ of $\O$ has characteristic zero, resp.\ the absolute Galois group of $k$ is a pro-$p$-group if $\kappa$ has characteristic $p>0$.
  \item $(A,A^+)$ is called \emph{strictly henselian} if $A$ is a strictly henselian ring. In this case $\O$, and, by \cref{triple-henselian}, also $A^+$ is strictly henselian.
\end{enumerate}
\end{definition}

\begin{lemma} \label{adic-hensel-char} Let $(A,A^+)$ be a Huber pair such that $\cX=\Spa(A,A^+)$ is connected. Then
\begin{enumerate}[\rm (i)]
\item $(A,A^+)$ is local if and only if every open covering of $\cX$ splits.
\item $(A,A^+)$ is strongly henselian if and only if every strongly \'{e}tale covering of $\cX$ splits.
\item $(A,A^+)$ is tamely henselian if and only if every tame covering of $\cX$ splits.
\item $(A,A^+)$ is strictly henselian if and only if every \'{e}tale covering  of $\cX$ splits.
\end{enumerate}
\end{lemma}

\begin{proof}
That the properties local, etc.\ imply the corresponding splitting property follows directly from the definitions.

Assume that every open covering of $\cX$ splits. Then  $\cX$ has exactly one closed point, say $v$, to which all other points specialize. Every open neighborhood of $v$ is equal to $\cX$, hence $(A,A^+)=(A,A^+)_v$ is local. This shows (i) and that $(A,A^+)$ is local in cases (ii)--(iv). We denote the image of $A^+$ in $k=A/\m_A$ by $\O$ (a valuation ring of $k$).

Now assume that every strongly \'{e}tale covering of $\cX$ splits. We first show that $A$ is henselian. Let $\Spec B \to \Spec A$ be a connected \'{e}tale covering such that there is a point $y\in \Spec B$ lying over $\supp v$ with trivial residue field extension $k(y)=k(\supp v)$. Let $B^+$ be the integral closure of $A^+$ in $B$. Then the natural projection $f: \Spa(B,B^+)\to \cX$ is \'{e}tale and strongly \'{e}tale at every point with support~$y$. Let $\cY\subset \Spa(B,B^+)$ be the strongly \'{e}tale locus of $f$, which is open by \cite[Corollary\,4.3]{HueAd}. Then $f(\cY)$ is open in $\cX$ and contains $v$, hence $f(\cY)=\cX$. We conclude that $f|_{\cY}: \cY \to \cX$ is a strongly \'{e}tale covering, hence admits a splitting. This induces a splitting of $\Spec B \to \Spec A$. We conclude that $A$ is a local henselian ring. In order to show (ii), it remains to show that $\O$ is strictly henselian.

If $\O$ were not strictly henselian,  there would exist a connected \'{e}tale $\O$-algebra $\O'$ which does not split. Being normal and connected, $\O'$ is a domain. Let $k'$ denote its quotient field. Then $k'/k$ is a nontrivial finite separable extension (and $\O'$ is a localization of the normalization of $\O$ in $k'$). Let $A\to B$ be the finite \'{e}tale extension of local henselian rings associated with $k'/k$. Let $\O'=\O[\bar{x}_1, \ldots, \bar{x}_n]$, $x_1,\ldots,x_n$ preimages of the $\bar{x}_i$ in $B$, and $B^+\subset B$ the integral closure in $B$ of $A[x_1,\ldots,x_n]$. Then $\Spa(B,B^+) \to \cX$ is \'{e}tale and strongly \'{e}tale over the closed point $v$. Since the strongly \'{e}tale locus is open, we obtain a strongly \'{e}tale covering of $\cX$ which does not split. This shows (ii).

Finally, using (ii), the verification of (iii) and (iv) is easy and left to the reader.
\end{proof}

For a Huber pair $(A,A^+)$ and a point $ x \in \Spa(A,A^+)$ the \emph{henselization} of $(A,A^+)$ at~$x$ is obtained as follows:
Let~$A^h$ be the henselization of~$A$ at $\supp x$. In its residue field~$k$ we have the valuation ring~$\O$ corresponding to~$x$.
The henselization $\O^h$ of $\O$ is a valuation ring, its quotient field $k^h$ is a separable algebraic field extension of $k$.
We define $C$ as the ind-(finite \'{e}tale) $A^h$-algebra corresponding to $k^h/k$ and~$C^+$ as the preimage of~$\O^h$ in~$C$.

\begin{definition}
We call  $(A,A^+)^h_x:=(C,C^+)$ the \emph{henselization} of $(A,A^+)$ at $x$.
\end{definition}

The induced homomorphism $(A,A^+)_x \to (A,A^+)^h_x$ is universal for local homomorphisms of $(A,A^+)_x$ to henselian Huber pairs.

Now we fix a separable closure $\bar k/k^h$. The henselian valuation of $k^h$ associated with $\O^h$ has a unique extension to $\bar k$.  Inside this we have the maximal tamely ramified and unramified extensions $k^t/k^\nr/k^h$. The normalization of $\O^h$ in $k^\nr$ is the strict henselization $\O^\sh$ of $\O^h$.  We denote the
normalization of $\O^h$ in $k^t$ by $\O^t$ and call it the tame henselization. Denoting the normalization of $\O^h$ in $\bar k$ by $\bar \O$, we obtain local ring extensions $\O^h\hookrightarrow \O^\nr\hookrightarrow\O^t\hookrightarrow \bar \O $.

\begin{definition}
The \emph{strong henselization} $(B,B^+)=(A,A^+)_x^\sh$ is given by setting $B$ the unique ind-(finite \'{e}tale) subextension of $A^\sh/A^h$ (strict henselization and henselization of $A$ at $\supp x$) associated to $k^\nr/k$ and $B^+$ the preimage of $\O^\sh \subset k^\nr$ in $B$.

The \emph{tame henselization} $(B,B^+)=(A,A^+)_x^t$ is given by setting $B$ the unique ind-(finite \'{e}tale) subextension of $A^\sh/A^h$ associated to $k^t/k$ and $B^+$ the preimage of $\O^t\subset k^t$ in $B$.

The \emph{strict henselization} $(B,B^+)=(A,A^+)_x^\strh$ is given by setting $B=A^\sh$ (strict henselization of $A$ at $\supp x$) and $B^+$ the preimage of $\bar \O \subset \bar k$ in $B$.
\end{definition}

The homomorphisms
\[
(A,A^+)_x^h \to (A,A^+)_x^\sh \to (A,A^+)_x^t \to (A,A^+)^\strh
\]
are integral by \cref{ind-etale-hens}\,(ii) below.

\begin{lemma}\label{m in plus}
Let $(A,A^+)$ be a henselian Huber pair and $(A,A^+)\to (B,B^+)$  a homomorphism of Huber pairs with $A\to B$ integral, ind-\'{e}tale and local. Then $\m_B\subset B^+$.
\end{lemma}

\begin{proof}
Every element of $x\in \m_B$ satisfies an equation of the form $x^n+a_{n-1}x^{n-1}+\cdots + a_0=0$ with $a_0,\ldots,a_{n-1}\in \m_A$ (the $a_i$ are up to sign the elementary symmetric functions in the conjugates of $x$ in some Galois closure of $B/A$). Since $\m_A\subset A^+$, $x$ is integral over $A^+$, hence $x\in B^+$.
\end{proof}

\begin{lemma} \label{ind-etale-hens} Let $f: (A,A^+)\to (B,B^+)$ be a local homomorphism of local Huber pairs.
\begin{enumerate}[\rm (i)]
\item  If $(A,A^+)$ is henselian (resp.\ strongly, tamely or strictly henselian) and $f$ is integral, then $(B,B^+)$ is henselian (resp.\ strongly, tamely or strictly henselian).
\item If $f$ is ind-\'{e}tale and $(A,A^+)$ and $ (B,B^+)$ are henselian, then $f$ is integral.
\end{enumerate}
\end{lemma}

\begin{proof}
(i) $B^+$, being an integral, local algebra over the henselian ring $A^+$, is henselian, as well. If $A^+$ is strictly henselian, then also $B^+$ is strictly henselian.  Let $k_A$ be the residue field of $A$ and $k_B$ that of $B$. Then $k_B/k_A$ is an algebraic field extension. We obtain an injection $G_{k_B}\subset G_{k_A}$ of absolute Galois groups showing the statement also in the strict and tame case.

(ii) We may assume that $f$ is \'{e}tale. Since $f$ is local and $A$ is henselian, $B$ is a local, finite \'{e}tale $A$-algebra. Let $A'$ be the integral closure of $A^+$ in $B$. We have to show that the inclusion $A' \to B^+$ is surjective. By \cref{m in plus}, we have  $\m_B\subset A'$. It therefore suffices to show that $A'/\m_B \to B^+/\m_B$ is surjective. Both are subrings of the field $B/\m_B=k_B$, which is a finite separable extension of $k_A=A/\m_A$.

We claim that $A'/\m_B$ is the integral closure of $A^+/\m_A\subset k_A$ in $k_B$. By definition of $A'$, every element of $A'/\m_B$ is integral over $A^+/\m_A$. On the other hand, if for $y\in B$ the element $\bar{y}\in k_B$ is integral over $A^+/\m_A$, then there is a congruence $y^n+a_{n-1}+\cdots + a_0\equiv 0 \bmod \m_B$ with all $a_i\in A^+$. Since $\m_B\subset A'$, $y$ is integral over $A'$ and hence $y \in A'$.

By assumption, $A^+/\m_A$ is a henselian valuation ring, hence so is $A'/\m_B$. Therefore  $A'/\m_B\to B^+/\m_B$ is a local homomorphism of valuation rings with the same quotient  field, hence an isomorphism. This concludes the proof.
\end{proof}

Now we consider Huber pairs $(A,A^+)$ via their adic spectrum $\Spa(A,A^+)$.

\begin{definition}
A Huber pair $(A,A^+)$ is \emph{\'{e}tale} (\emph{strongly \'{e}tale}, \emph{tamely}) \emph{acyclic} if every \'{e}tale (strongly \'{e}tale, tame) covering of $\Spa(A,A^+)$ splits.
\end{definition}

\begin{proposition}\label{tamely acyclic char} A Huber pair
$(A,A^+)$ is \'{e}tale (strongly \'{e}tale, tamely) acyclic if and only if every connected component of $\Spa(A,A^+)$ is the adic spectrum of a strictly (strongly, tamely) henselian Huber pair.
\end{proposition}

\begin{proof} We consider the canonical inclusion $\Spec A \hookrightarrow \Spa(A,A^+)$, $x\mapsto (x,\tr_x)$. Then $\Spec A$ carries the subspace topology and is dense in $\Spa(A,A^+)$. Hence the connected components of $\Spa(A,A^+)$ are in bijection with the connected components of $\Spec A$.

Let $\p\subset A$ be a prime ideal and let $e$ run through all idempotents of $A$ contained in $\p$. Then the  connected component of $\p \in \Spec A$  is equal to $\Spec B$ with $B= \varinjlim_e A/e$. Note that $B$ can also be obtained from $A$ by inverting all elements  $1-e$.
Since $A^+$ is integrally closed in $A$, all idempotents of $A$ are in $A^+$.
Therefore the components of $\Spa(A,A^+)$ are of the form $\Spa(B,B^+)$, where $B=\varinjlim_e A/e$ is a component of $A$ and $B^+=\varinjlim_e A^+/e$.

Now the statement of the proposition reduces to \cref{adic-hensel-char} by exactly the same limit argument as in the proof of the analogous statement for schemes \cite[Proposition 3.2]{Art71}.
\end{proof}

We call a Huber pair $(A,A^+)$ \emph{quasi-acyclic} if every connected component of $\Spa(A,A^+)$ is the adic spectrum of a henselian Huber pair.

\section{Joins of henselian Huber pairs}

Our next aim is to compare the tame cohomology of an $S$-scheme $X$ with the tame cohomology of its associated adic space $\Spa(X,S)$. This is rather easy for \v{C}ech cohomology and we have proven in \cref{sec:cech} that tame cohomology coincides with tame \v{C}ech cohomology for many schemes. Hence our next task is to prove a similar statement for discretely ringed adic spaces. This turns out to be more complicated as the analysis of joins is more involved.

\bigskip
Let $(A,A^+)$ be a Huber pair such that $A$ is a normal domain. Let $K=Q(A)$ and $\bar{K}/K$ a separable closure. Let $v_1,v_2\in \Spa(A,A^+)$ be points and $(B_i,B_i^+)=(A,A^+)_{v_i}^h$. We embed $B_1$ and $B_2$ into $\bar{K}$, hence also $B_1^+$ and $B_2^+$ are embedded into $\bar{K}$.  Let $D=[B_1,B_2]\subset \bar K$ be the subring over $A$ generated by $B_1$ and $B_2$ and let $D^+\subset D$ be the integral closure of $[B_1^+,B_2^+]$ in $D$. We call the Huber pair $(D,D^+)$ the \emph{join} of the Huber pairs $(B_1,B_1^+)$ and $(B_2,B_2^+)$ (with respect to the chosen embeddings into $\bar K$):
\[
\begin{tikzcd}
  &\bar{K}\arrow[d,dash]\\
  &(D,D^+)=[(B_1,B_1^+),(B_2,B_2^+)]\arrow[dl,dash]\arrow[dr,dash]\\
(B_1,B_1^+)\arrow[dr,dash]&&(B_2,B_2^+)\arrow[dl,dash]\\
&(A,A^+)
\end{tikzcd}
\]

Now we can formulate the adic analogue of Artin's \cref{artin-2.5}.

\begin{theorem} \label{huberjoin} With the assumptions above, the join $(D,D^+)$ is a henselian Huber pair.
If $v_1,v_2$ are Riemann-Zariski points such that none of the $v_i$ specializes to the other, then $(D,D^+)$ is the strong henselization of $(A,A^+)$ at some Riemann-Zariski point $w\in \Spa(A,A^+)$: $(D,D^+)=(A,A^+)_w^\sh$.
\end{theorem}

We start with some lemmas.

\begin{lemma} \label{affinoid_open_immersion}
 Let $(A,A^+)$ be a Huber pair such that $A$ is of finite type over $A^+$.
 Then any finite subset $M$ of $\Spa(A,A^+)$ has a common affinoid neighborhood $\Spa(B,B^+)$ such that $\Spec B^+ \to \Spec B$ is an open immersion.
 Moreover, any open affinoid contained in $\Spa(B,B^+)$ has this property.
\end{lemma}

\begin{proof}
 Let $\Spec A \stackrel{i}{\hookrightarrow} S \stackrel{p}{\to} \Spec A^+$ be a compactification with an open immersion $i$ and $p$ projective. Let $\{c_m\}_{m\in M}$ be the set of centers in $S$ of the points in $M$. Then we can find a common open affine neighborhood $\Spec B'$ of the $c_m$'s.
 Setting $B = B' \otimes_{A^+} A$ and $B^+$ the normalization of $B'$ in $B$, we obtain the required affinoid neighborhood $\Spa(B,B^+)$.

 For an affinoid open $\Spa(C,C^+) \subset \Spa(B,B^+)$ consider the diagram
 \[
  \begin{tikzcd}
   \Spec C		\ar[r]	\ar[d]	& \Spec B	\ar[d]	\\
   \Spec C^+	\ar[r]			& \Spec B^+.
  \end{tikzcd}
 \]
 The upper horizontal and the right vertical arrows are open immersions.
 Hence, so is the left vertical arrow.
\end{proof}

\begin{lemma} \label{affinoid_separate_centers}
 Let~$M$ be a finite set of Riemann-Zariski points of an affinoid adic space $\Spa(A,A^+)$. Assume that no element of $M$ specializes to another.
 Then among all open affinoid neighborhoods $\Spa(B,B^+)$ of $M$ those are cofinal where the centers $c_m$ of the elements $m \in M$ in $\Spec B^+$ do not satisfy any specialization relation.
\end{lemma}

\begin{proof}
 Consider the map
 \[
 c: \Spa(A,A^+)\longrightarrow \lim S,
 \]
 where $S$ runs through all $\Spec A$-modifications of $\Spec A^+$, i.e., factorizations $$\Spec A \stackrel{i}{\hookrightarrow} S \stackrel{p}{\to} \Spec A^+$$ with $i$ an open dense immersion and $p$ proper, and the map $c$ is given by sending a valuation to the system of its centers on the various $S$.
 Let $\RZ(A,A^+)$ be the subspace of all Riemann-Zariski points in $\Spa(A,A^+)$. By \cite[Corollary 3.4.7]{Tem11}, the restriction of $c$ to $\RZ(A,A^+)$ is a homeomorphism  $$\RZ(A,A^+)\stackrel{\sim}{\longrightarrow} \lim S.$$ By Temkin's generalization of Chow's Lemma, \cite[Corollary 3.4.8]{Tem11}, among the modifications, those with $p: S\to \Spec A^+$ projective are cofinal. Since the $m \in M \subset \RZ(A,A^+)$ do not specialize to each other, we find a projective $\Spec A$-modification $S\to \Spec A^+$ such that the centers $(c_m)_{m\in M}$ on $S$ do not specialize to each other.
 Now choose a common open affine neighborhood $\Spec B' \subset S$ of the~$c_m$.
 Set $B = B' \otimes_{A^+} A$ and $B^+$ the normalization of $B'$ in $B$.
 This provides the required open affinoid neighborhood $\Spa(B,B^+)$.
 It is clear that any smaller open affinoid neighborhood of $M$ has the same property.
\end{proof}

\begin{lemma}\label{valjoin} Let $K$ be a field.
\begin{enumerate}[\rm (i)]
  \item If $\O\subset K$ is a henselian valuation ring of $K$ and  $\O'\supset \O$ is an overring of $\O$ in $K$, then also $\O'$ is a henselian valuation ring.
  \item If $\O,\O'\subset K$ are henselian valuation rings of $K$ such that none of $\O,\O'$ contains the other, then their join $\O\O'$ is strictly henselian.
\end{enumerate}
\end{lemma}

\begin{proof}
(i) Being an overring of $\O$, $\O'$ is a valuation ring. Let us denote the associated valuations by $v,v'$. Then $v'$ is a generalization of $v$ and $\O'$ is the localization of $\O$ at some prime ideal $\p$.
In order to show that $v'$ is henselian, we have to show that there is at most one extension of $v'$ to $L$ in any finite separable field extension $L/K$. Let $A$ be the normalization of $\O$ in $L$. Since $\O$ is henselian, $A$ is local, hence a valuation ring of $L$. The extensions of $v'$ to $L$ are in bijection to the prime ideals of $A$ lying over $\p\subset \O$. But there is only one since there are no containment relations between those primes and the set of primes of the valuation ring $A$ is totally ordered with respect to inclusion.

(ii) Let $\m\subset \O$ and $\m'\subset \O'$ be the maximal ideals. Put $A=\O\cap \O'$ and $\p=\m\cap A$, $\p'=\m'\cap A$.
By \cite[Lemma 3.2.6]{EP2005}, we have $\O=A_\p$, $\O'=A_{\p'}$. Since none of $\O,\O'$ contains the other, none of $\p,\p'$ contains the other. Hence the statement follows from \cref{artin-2.5}.
\end{proof}

\begin{proof}[Proof of \cref{huberjoin}]  By \cref{artinslemma}, $D$ is local, henselian.

\smallskip\noindent
Assume we are in case (1)  of  \cref{artinslemma}. Then both ring homomorphisms $B_1\to D$, $B_2\to D$ are not local and $k$ is separably closed by loc.\ cit. Let $\m_1$, $\m_2$ be the maximal ideals of $B_1$ and $B_2$. Since the homomorphisms are not local, the ideal $\m_1\m_2 \subset D$ is not contained in $\m_D$. Hence it is the whole ring $D$. Since $B_1^+$ and $B_2^+$ contain $\m_1$ and $\m_2$, respectively, we conclude that $D^+=D$. Hence $(D,D^+)$ is strongly henselian (it is the strong henselization of $(A,A^+)$ at some trivial valuation).

\smallskip\noindent
Assume we are in case (2) of  \cref{artinslemma} and, say, $B_1\to D$ is a local, integral and ind-\'{e}tale ring homomorphism. By \cref{m in plus}, we have $\m_{D}\subset D^+$. Let $\O:= D^+/\m_{D}\subset k=D/\m_{D}$ and $\O_1=B_1^+/\m_1\subset k_1=B_1/\m_1$. Then $\O$ contains the integral closure $\O_1'$ of $\O_1$ in the field extension $k/k_1$. Since $\O_1$ is a henselian valuation ring, the same is true for $\O_1'$. Hence $\O$ is an overring of a henselian valuation ring of $k$, hence itself a henselian valuation ring by \cref{valjoin}(i).  For $x\in D$ with $x \bmod \m_B \in \O$, we find $y\in D^+$ with $x-y\in \m_{D}\subset D^+$. This shows that $D^+$ is the full pre-image of $\O$ in~$D$. Hence $(D,D^+)$ is local. We conclude that $(D,D^+)$ is a henselian Huber pair.

\smallskip\noindent
For the rest of the proof we  assume that $v_1,v_2$ are Riemann-Zariski and that we are in case (2) of  \cref{artinslemma}.
Let $w'$ be the unique closed point of $\Spa(D,D^+)$ and let $w\in \Spa(A,A^+)$ be its image.
Let $\bar A$ and $\bar A^+$ be the integral closures of $A$ and $A^+$ in $\bar K$. Let $v_i'$ be the unique closed point of $(B_i,B_i^+)$. Then $v_i'$ maps to $v_i$ in $\Spa(A,A^+)$. On the other hand, since $(B_i,B_i^+)$ is henselian, there is a unique point $\bar v_i\in \Spa(\bar A, \bar A^+)$ lying over $v_i'$.

By assumption, none of $\bar v_1$, $\bar v_2$ specializes to each other.
The ring $D$ is a component of the tensor product $B_1\otimes_A B_2$ (and every component of the tensor product arises in this way).
It remains to show that $D^+\subset D$ is the pre-image of a strictly henselian valuation ring $\O\subset k=D/\m_{D}$.
Since $\Spa(D,D^+)$ is henselian, there is a unique point $\bar w \in \Spa(\bar A,\bar A^+)$ over $w'$. We learn:

 \medskip\noindent
 1. Since  $(A,A^+)\to (D,D^+)$ is ind-strongly \'{e}tale and $(D,D^+)$ is henselian, $(D,D^+)$ lies somewhere between $(A,A^+)_w^h$ and $(A,A^+)_w^\sh$.  (We want to show that $(D,D^+)=(A,A^+)_w^\sh$.)

 \noindent
2. $\bar w$ is a vertical generalization of $\bar{v}_1$ and a generalization of $\bar{v}_2$.

\noindent
3. We are in the case that $B_1\to D$ is local, hence  (cf.\ the proof of \cref{artinslemma}) in the case $\supp \bar v_1 \rightsquigarrow \supp \bar v_2$ in $\Spec(\bar A)$.

\medskip\noindent
If $\supp \bar{v}_1=\supp \bar{v}_2$, then the statement reduces to the case of valuations rings, where it follows from \cref{valjoin}.
Hence we may assume that $\supp \bar{v}_1\subsetneqq \supp \bar{v}_2$.
Since $v_1$ is a Riemann-Zariski point, the same is true for $\bar{v}_1$ and hence $\bar w$ is a Riemann-Zariski point by \cref{vertical-RZ}(ii). But $\bar v_2$ is a specialization of $\bar w$. By \cref{specialization-prop},  this is only possible if $\bar w= \tr_{\supp \bar w}$ and the specialization $\bar w \rightsquigarrow \bar v_2$ is of the form $\tr_{\supp \bar w} \rightsquigarrow \tr_{\supp \bar v_2} \rightsquigarrow \bar v_2$, where the second specialization is vertical.

We learn that $D=D^+$ and have to show that this henselian ring is strictly henselian. By a straightforward limit argument, we may assume that $A$ is finitely generated over $\Z$, in particular, $A$ is Nagata. Using \cref{formallyreal}, we may replace $A$ by the normalization of $A/\supp v_1$, i.e., we may assume that $\supp v_1$ is the generic point of $\Spec A$ and $D=D^+$ is a field. Since $v_1$ does not specialize to $v_2$, $v_1$ is not the trivial valuation (moreover, in this case the assertion would be true by trivial reasons).

Let $\bar v_3$ be the trivial valuation with support $\supp \bar v_2$, $v_3$ the restriction of $\bar v_3$ to $A$ and $(B_3,B_3^+)=(A,A^+)_{v_3}^h$. Then $B_2$ is an integral \'{e}tale extension of $B_3=B_3^+$. In order to show the assertion, we therefore may replace $v_2$ by the trivial valuation~$v_3$.

\bigskip\noindent
After all these reductions, we arrive at the following situation:\medskip

$\bullet$ $A$ is a normal domain of finite type over $\Z$ with quotient field $K$

$\bullet$ $\supp v_1$ is the generic point of $\Spec A$, $(B_1,B_1^+)=(A,A^+)_{v_1}^h$

$\bullet$ $B_1$ is a field (intermediate extension of $\bar K/K$) and $B_1^+\subsetneqq B_1$ a valuation ring

$\bullet$ $v_2$ is a trivial valuation, $(B_2,B_2^+)=(A,A^+)_{v_2}^h$, i.e., $B_2=B_2^+=A_{\supp v_2}^h$

$\bullet$ $D=D^+$ is a field (intermediate extension of $\bar K / B_1$).

\medskip\noindent
Furthermore, by \cref{affinoid_open_immersion} we may assume that $\Spec A \to \Spec A^+$ is an open immersion and by \cref{affinoid_separate_centers} we may assume that the centers $c_1,c_2 \in \Spec A^+$ of $v_1, v_2$ do not specialize to each other.

\medskip\noindent
Moreover, by \cref{formallyreal}, we are still permitted to replace $(A,A^+)$ by its normalization in a finite field extension $L/K$ inside $\bar K$ (and replace $v_1$, $v_2$ by points over them).

\medskip
The local homomorphisms $A_{c_1}^+ \to \O_{v_1}$, $A_{c_2}^+ \stackrel{\sim}{\to} A_{\supp v_2}$ induce maps $A_{c_1}^{+h} \to B_1^+$, $A_{c_2}^{+h}\stackrel{\sim}{\to} B_2^+$, hence a homomorphism
\[
 \ C:=[A_{c_1}^{+h}, A_{c_2}^{+h} ] \hookrightarrow D=[B_1,B_2] .
\]
By \cref{artin-2.5}, $C$ is strictly henselian. Since $C$ is the inductive limit of \'{e}tale $A^+$-algebras, it is the strict henselization of $A^+$ at some prime $\p^+$.
As $A \subset A_{\supp v_2} = A_{c_2}^+ \subset C$, the morphism $\Spec C \to \Spec A^+$ factors through the open subset $\Spec A$ of $\Spec A^+$. Hence $C$ can also be obtained as the strict henselization of $A$ at $\p = \p^+A$. In particular, $C$ is noetherian and has finite Krull-dimension.  If $C$ were a (separably closed) field, the inclusions $A\subset C\subset D\subset \bar K$ would prove that $D$ is separably closed, hence our assertion.  We will achieve this by decreasing induction on $\dim C$. Assume $\dim C >0$. Then $\p$ is not the zero ideal, i.e., we find $0\ne x \in \p \subset A$.

The support of $v_1$ is the generic point of $\Spec A$, hence $\O_{v_1}\subset K$. By valuation theory, the join $[\O_{v_1}, A]\subset K$ is the localization of $\O_{v_1}$ at some prime ideal. If $[\O_{v_1}, A]$ were a proper subring of $K$, \cref{special-aff-chara}(ii) would imply the existence of   a nontrivial horizontal specialization of $v_1$. Since  $v_1$ is Riemann-Zariski, we obtain
\[
K= [\O_{v_1}, A].
\]
Hence we can write
\[
x^{-1}= \sum_{\text{finite}} y_iz_i, \ y_i \in \O_{v_1},\ z_i\in A.
\]
Now the homeomorphism $\RZ(A,A^+) \cong \lim S$ above respects stalks. This means for a Riemann-Zariski point $v\in \Spa(A,A^+)$ with $(A,A^+)_v=(E,E^+)$, the ring $E^+$ is the colimit of the (scheme-) stalks of the models at the respective centers, i.e., $E^+=\colim \O_{S,c(v)}$. Applying this to $v_1$, we can (in the manner as we did above), modify $(A,A^+)$ to $(A',A'^+)$ to achieve that all $y_i$ lie in $A'_{c_1}$. Defining $C'$ similar to $C$ but for $(A',A'^+)$, we obtain the commutative diagram
\[
\begin{tikzcd}
  &D=D^+\\
  C\arrow[ru,hook]\arrow[rr,hook]&&C'\arrow[lu, hook]\\
  A\arrow[u,hook]\arrow[rr,hook]&&A'\arrow[u,hook],
\end{tikzcd}
\]
in which the lower horizontal map induces an open immersion on $\Spec$. Since all $y_i,z_i$ are in $C'$, $x$ is invertible in $C'$, hence the image of the closed point of $\Spec C'$ in $\Spec A$ is a nontrivial generalization of $\p$. This process stops since $A$ has finite Krull-dimension.
\end{proof}

Recall that the tensor product of Huber pairs $$(B,B^+)= (B_1,B_1^+)\otimes_{(A,A^+)} \ldots \otimes_{(A,A^+)} (B_n,B_n^+)$$ is defined as follows: $B$ is the (usual) tensor product $B_1\otimes_A \cdots \otimes_A B_n$ and $B^+$ is the integral closure of $B_1^+\otimes_{A^+} \cdots \otimes_{A^+} B_n^+$ in $B$.

\begin{theorem}\label{hubertensor}
Let $(A,A^+)$ be a Huber pair, $v_1,v_2\in \Spa(A,A^+)$ and, for $i=1,2$,  $(B_{i},B_i^+)$ a local, integral $(A,A^+)_{v_i}^h$-algebra.
Then
\[
(B,B^+)= (B_{1},B_1^+) \otimes_{(A,A^+)} (B_{2},B_2^+)
\]
is quasi-acyclic. Assume that $v_1,v_2\in \Spa(A,A^+)$ are  Riemann-Zariski points and let $(D,D^+)$ be a component of $(B,B^+)$. Then the closed point of $(D,D^+)$ maps to a Riemann-Zariski point of $\Spa(A,A^+)$ and one of the following holds.
\begin{enumerate}
  \item If both homomorphisms $(B_i,B_i^+) \to (D,D^+)$ are not local, then the residue field $D^+/\m_{D^+}$ of $D^+$ is separably closed.
  \item If, say, $(B_1,B_1^+)\to (D,D^+)$ is local, it is integral. In particular, $D^+/\m_{D^+}$ is an algebraic field extension of $B_1^+/\m_{B_1^+}$.
\end{enumerate}
\end{theorem}

\begin{proof} This follows in exactly the same way from \cref{huberjoin} as \cref{artinslemma} follows from \cref{artin-2.5}.
\end{proof}

We call a  homomorphism of Huber pairs $f: (A,A^+)\to (B,B^+)$ \emph{ind-\'{e}tale}
if it is a filtered colimit of \'{e}tale homomorphisms $f_i: (A,A^+)\to (B_i,B^+_i)$.

\begin{theorem}\label{product of acyclic Huber pairs}
Let $(A,A^+)$ be a Huber pair and let $(B_1,B_1^+),\ldots, (B_n,B_n^+) $ be ind-\'{e}tale $(A,A^+)$-algebras which are quasi-acyclic.  Then the tensor product
\[
(B,B^+)= (B_1,B_1^+)\otimes_{(A,A^+)} \ldots \otimes_{(A,A^+)} (B_n,B_n^+)
\]
is an ind-\'{e}tale quasi-acyclic $(A,A^+)$-algebra.  Suppose that all closed points of the $\Spa(B_i,B_i^+)$ map to Riemann-Zariski points in $\Spa(A,A^+)$. Then the closed points of $\Spa(B,B^+)$ map to Riemann-Zariski points in $\Spa(A,A^+)$.
For a component $(D,D^+)$ of $(B,B^+)$ one of the following holds
\begin{enumerate}
  \item $D^+/\m_{D^+}$ is separably closed, or
  \item $D^+/\m_{D^+}$ is an algebraic extension of $D_i^+/\m_{D_i^+}$, where  $(D_i,D_i^+)$ is a component of $(B_i,B_i^+)$ for some $i$, $1\leq i\le n$.
\end{enumerate}
\end{theorem}

\begin{proof} Regarding \cref{ind-etale-hens}(ii), this follows in exactly the same way from \cref{hubertensor} as \cref{Artins-thm} follows from \cref{artinslemma}.
\end{proof}

\begin{corollary}\label{product of strongly acyclic Huber pairs}
Let $(A,A^+)$ be a Huber pair and let $(B_1,B_1^+),\ldots, (B_n,B_n^+) $ be ind-\'{e}tale quasi-acyclic $(A,A^+)$-algebras.
Then the tensor product
\[
(B,B^+)= (B_1,B_1^+)\otimes_{(A,A^+)} \ldots \otimes_{(A,A^+)} (B_n,B_n^+)
\] is quasi-acyclic.
\begin{enumerate}[\rm (i)]
  \item If all $\Spa(B_i,B_i^+)$ are \'{e}tale acyclic, then also $\Spa(B,B^+)$ is \'{e}tale acyclic.
  \item If all $\Spa(B_i,B_i^+)$ are strongly \'{e}tale acyclic and all morphisms $\Spa(B_i,B_i^+)\to (A,A^+)$ map closed points to Riemann-Zariski points, then $(B,B^+)$ is strongly \'{e}tale acyclic.
  \item If $A$ is an $\F_p$-algebra for some prime number $p$, all $(B_i,B_i^+)$ are tamely acyclic and all morphisms $\Spa(B_i,B_i^+)\to (A,A^+)$ map closed points to Riemann-Zariski points, then
  $(B,B^+)$ is tamely acyclic.
\end{enumerate}
\end{corollary}

\begin{proof}
The first assertion follows immediately from \cref{product of acyclic Huber pairs}. Assertion (i) follows from \cref{Artins-thm} and (ii) follows again from \cref{product of acyclic Huber pairs}.
It remains to show (iii).
Assume that $A$ is an $\F_p$-algebra and let $(D,D^+)$ be a component of $(B,B^+)$. Since we already know that $(D,D^+)$ is strongly henselian, it suffices to show that the absolute Galois group $G_k$ of $k=D/\m_D$ is a pro-$p$-group. Note that $D$ is a component of $B$. Hence, by \cref{Artins-thm}, either $k$ is separably closed (in which case we are done) or there is a component $D_i$ of $B_i$ for some $i$ such that $D_i\to D$ is local and $k$ is an algebraic extension of the residue field $D_i/\m_{D_i}$. The absolute Galois group of $D_i/\m_{D_i}$ is a pro-$p$-group by assumption. Therefore its subgroup $G_k$ is also a pro-$p$-group.
\end{proof}

\section{Riemann-Zariski morphisms}

The main technical problem in proving a comparison theorem between the sheaf cohomology and \v{C}ech cohomology of  discretely ringed adic spaces along the lines of the proof for schemes as given in \cref{sec:cech} is the assumption on Riemann-Zariski points in \cref{product of acyclic Huber pairs} and \cref{product of strongly acyclic Huber pairs}. Therefore we have to carry out  a more thorough analysis of Riemann-Zariski points first.

Let $\cX$ be a discretely ringed adic space. We say that $P\in \cX$ specializes to $Q\in \cX$ (notation: $P \rightsquigarrow Q$) if $Q \in \overline{\{P\}}$. Let $\cX_P$ be the local adic space of $\cX$ in $P$, i.e., $\cX_P= \Spa (A,A^+)_P$ for an open affinoid neighbourhood $\Spa(A,A^+)$ of $P$ in $\cX$.    Then $P\rightsquigarrow Q$ if and only if
the morphism $\cX_P \to \cX$ factors through $\cX_Q\to \cX$.

\begin{definition} Let $P\rightsquigarrow Q$. We say $Q$ is a vertical resp.\ horizontal specialization of $P$ if this is the case in one/every affinoid open of $\cX$ containing $P$ and $Q$.
A point $x \in \cX$ is called a \emph{Riemann-Zariski point} if it has no nontrivial horizontal specializations.

We call a morphism $f: \cX' \to \cX$  \emph{Riemann-Zariski} if for every Riemann-Zariski point $x'\in \cX'$, $f(x')$ is a Riemann-Zariski point in $\cX$.
\end{definition}

In \cref{sec:cont} we recalled Temkin's construction of the discretely ringed adic space $\Spa(X,S)$ associated to a scheme morphism $X\to S$. Points of $\Spa(X,S)$ are triples $(x,v,\varepsilon)$, where~$x$ is a point of~$X$,~$v$ is a valuation of~$k(x)$ and $\varepsilon : \Spec \O_v \to S$ is a morphism compatible with $\Spec k(x) \to S$.

\begin{lemma} \label{hori-char} Let $p: X\to S$ be a scheme morphism,  $\cX= \Spa(X,S)$ and $(x,v,\varepsilon), (y,w,\mu)\in \cX$.
\begin{enumerate}[\rm (i)]
  \item $(y,w,\mu)$ is a vertical specialization of $(x,v,\varepsilon)$ if and only if $x=y$, $\O_w \subset \O_v\ (\subset k(x))$, and the diagram
  \[
\begin{tikzcd}
\Spec \O_v \arrow[d]\arrow[rd,"\varepsilon"]\\
\Spec \O_w \arrow[r,"\mu"']&S
\end{tikzcd}
\]
commutes (automatic if $S$ is separated).
  \item $(y,w,\mu)$ is a horizontal specialization of $(x,v,\varepsilon)$ if and only if there exists a point $P \in \Spec \O_v$ and an $x$-morphism $f:(\Spec \O_v)_P \to X$
  such that
  \begin{enumerate}[(a)]
    \item $f(P)=y$.
    \item The valuation $w$ is the restriction to $k(y)$ of the valuation on $k(P)$ given by the valuation ring $ \O_v/P $.
    \item The diagram (which exists by $(b)$ since $\O_v/P$ is the ring of global sections of the affine scheme $\overline{\{P\}}$.)
    \[
    \begin{tikzcd}
    \overline{\{P\}}\arrow[r]\arrow[d, hook]& \Spec \O_w\arrow[d,"\mu"]\\
    \Spec \O_v \arrow[r,"\varepsilon"]& S
    \end{tikzcd}
    \]
    commutes. (Note that (c) is automatic if $S$ is separated.)
   \end{enumerate}

Here is a diagram elaborating the conditions:
  \[
  \begin{tikzcd}
  \Spec \O_v \arrow[dddr,bend right=4cm,"\varepsilon"']&\overline{\{P\}} \arrow[l,hook']\arrow[rd]&\\
  (\Spec \O_v)_P \arrow[u,hook] \arrow[rd,"f"] & P\arrow[l,hook']\arrow[u,hook]\arrow[rd]&\Spec \O_w\arrow[ddl, bend left=3cm,"\mu"]\\
  x \arrow[u,hook] \arrow[r,hook]& X\arrow[d,"p"] & y\arrow[l,hook']\arrow[u,hook]\\
  &S&
    \end{tikzcd}
 \]
\end{enumerate}
\end{lemma}

\begin{proof}
(i) is obvious.\\
(ii) If $(y,w,\mu)$ is a horizontal specialization of $(x,v,\varepsilon)$, then both points lie in a common open affinoid and the conditions hold by \cref{special-aff-chara}. The same argument shows that the given conditions imply that $(y,w,\mu)$ is a horizontal specialization of $(x,v,\varepsilon)$ if we find a common open affinoid neighborhood of the two points. Since $f(P)=y$ by $(a)$, we have $x\rightsquigarrow y$ in $X$ and $p(x)\rightsquigarrow p(y)$ in $S$. Denoting the closed point of $\Spec \O_w$ by $c_w$, we moreover have $p(y) \rightsquigarrow \mu(c_w)$. Let $S'$ be an open affine neighborhood of $\mu(c_w)$ in $S$ and $X'$ an open affine neighborhood of $y$ in $p^{-1}(S')$. Then $\Spa(X',S')$ is an open affinoid neighborhood of $(y,w,\mu)$ in $\Spa(X,S)$. We are done if we can show that $(x,v,\varepsilon)$ lies in $\Spa(X',S')$. First, $x\rightsquigarrow y$ implies $x\in X'$. Hence it suffices to show that $\varepsilon: \Spec \O_v \to S$ factors through $S'$. For this it suffices to show $\varepsilon(c_v)\in S'$, where  $c_v$ is the closed point of $\Spec \O_v$. This follows since the closed point of $\overline{\{P\}}$ maps to $c_v$ and $c_w$ respectively and condition (c)  implies $\varepsilon (c_v)=\mu(c_w)$.
\end{proof}

\begin{corollary}\label{hor-spez.to}
If\/ $X$ and $S$ are separated, then the set of horizontal specializations of a point in $\Spa(X,S)$ is totally ordered.
\end{corollary}

\begin{proof} Let $(x,v,\varepsilon)\rightsquigarrow (y,w,\mu)$ be a horizontal specialization in $\Spa(X,S)$. We show that, with the notation of part (ii) of \cref{hori-char}, the triple $(y,w,\mu)$ is uniquely given by $P\in \Spec \O_v$.

First of all, since $X$ is separated, the morphism $f: (\Spec \O_v)_P \to X$, and hence $y=f(P)$ is unique.  Further, $\O_w= (\O_v/P) \cap k(y)$ (intersection in $k(P)$), hence $w$ is given by~$P$. Finally, since $S$ is separated, $\mu: \Spec \O_w \to S$ is uniquely given by~$y$, hence by~$P$.
The result follows since the set of prime ideals in a valuation ring is totally ordered.
\end{proof}

\begin{corollary} \label{characterize_RZ_points}
 A point $(x,v,\varepsilon)$ in $\Spa(X,S)$ is Riemann-Zariski if and only if $x$ is a closed point in $X \times_S \Spec \O_v$.
\end{corollary}

\begin{proof}
Assume that $x$ is not closed in $X \times_S \Spec \O_v$. Then there exists a nontrivial specialization  $\tilde y$ of $x$ in $X \times_S \Spec \O_v$. Hence we can find a valuation ring $\O\subsetneq k(x)$ such that $x\to X \times_S \Spec \O_v$ factors as $x \to \Spec \O \to X \times_S \Spec \O_v$ and $\tilde y$ is the image of the closed point in $\Spec \O$. Projecting to the second factor shows  $\O_v \subset \O$, hence $\O=(\O_v)_P$ for a point $P\in \Spec \O_v$. Projecting to the first factor, we find an $x$-morphism $f:\Spec \O \to X$ and hence by \cref{hori-char} a horizontal specialization of  $(x,v,\varepsilon)$. Since $\O \subsetneq k(x)$, the image of the closed point of $\Spec \O$ in $X$ is not $x$, hence the horizontal specialization of $(x,v,\varepsilon)$ is nontrivial. We conclude that $(x,v,\varepsilon)$ is not a Riemann-Zariski point.

If $(x,v,\varepsilon)$ is not a Riemann-Zariski point, it has a nontrivial horizontal specialization. Then the $x$-morphism $f: (\Spec \O_v)_P \to X$ of \cref{hori-char} and the canonical morphism $(\Spec \O_v)_P \to \Spec \O_v$ give a morphism $(\Spec \O_v)_P \to X\times_S \Spec \O_v$. It maps $P$ to a proper specialization of $x$.
\end{proof}

\begin{lemma} \label{RZ_specialization}
 If $X$ is quasi-compact, then every point in $\Spa(X,S)$ has a horizontal specialization which is a Riemann-Zariski point.
 Moreover, if $X$ and $S$ are separated, this Riemann-Zariski point is unique.
\end{lemma}

\begin{proof}
Let $(x,v,\phi)$ be a point in $\Spa(X,S)$. The set of horizontal specializations is partially ordered.  As~$X$ is quasi-compact, it is covered by finitely many affinoids $\Spa(A_i,A_i^+)$ and thus any chain of horizontal specializations of $(x,v,\phi)$ is contained entirely in one of the $\Spa(A_i,A_i^+)$. We can then view~$v$ as a valuation of~$A_i$ and all specializations of the chain specialize to $v|{c\Gamma_v}$. Hence, by Zorn's lemma, there is a maximal horizontal specialization $(y,w,\psi)$ of $(x,v,\phi)$. In particular, $(y,w,\psi)$ is a Riemann-Zariski point. If $X$ and $S$ are separated, $(y,w,\psi)$ is unique as the horizontal specializations are totally ordered by \cref{hor-spez.to}.
\end{proof}

\begin{definition}
 A commutative square of morphisms of schemes
 \[
  \begin{tikzcd}
   X'	\ar[r]	\ar[d]	& X	\ar[d]	\\
   S'	\ar[r]			& S
  \end{tikzcd}
 \]
 is said to have universally closed diagonal if the induced morphism $X' \to X \times_S S'$ is universally closed.
\end{definition}

\begin{lemma} \label{RZ_closed_diagonal}
 Let
 \[
  \begin{tikzcd}
   X'	\ar[r]	\ar[d]	& X	\ar[d]	\\
   S'	\ar[r]			& S
  \end{tikzcd}
 \]
 be a square with universally closed diagonal.
 Then the induced morphism
 \[
  \Spa(X',S') \to \Spa(X,S)
 \]
is Riemann-Zariski.  The converse holds if $X'$ is quasi-compact and all residue field extensions of $X' \to X$ are algebraic (e.g., if $X' \to X$ is \'{e}tale).
\end{lemma}

\begin{proof}
 Suppose the square has universally closed diagonal.
 Let $(x',v',\pi')$ be a Riemann-Zariski point of $\Spa(X',S')$ and assume that its image $(x,v,\phi)$ in $\Spa(X,S)$ is not a Riemann-Zariski point.
 This means there exists a localization $\O \subsetneq k(x)$ of the valuation ring $\O_v$ such that the dotted arrow exists in the following diagram:
 \[
  \begin{tikzcd}
   \Spec k(x)	\ar[rr]	\ar[d]			&								& X	\ar[d]	\\
   \Spec \O	\ar[r]	\ar[urr,dashed]	& \Spec \O_v	\ar[r,"\phi"]	& S.
  \end{tikzcd}
 \]
 There is a localization $\O' \subset k(x')$ of~$\O_{v'}$ with $\O' \cap k(x) = \O$.
 This yields a diagram
 \[
  \begin{tikzcd}
   \Spec \O'	\ar[drr]	\ar[dr,dashed]	\ar[ddr,"\phi'"']				\\
																& X \times_S S'	\ar[r]	\ar[d]	& X	\ar[d]	\\
																& S'			\ar[r]			& S.
  \end{tikzcd}
 \]
The valuative criterion for universal closedness for the morphism $X' \to X \times_S S'$ provides a lift:
 \[
  \begin{tikzcd}
   \Spec \O'	\ar[drr]	\ar[dr,dashed]	\ar[ddr,"\phi'"']				\\
																& X'	\ar[r]	\ar[d]	& X	\ar[d]	\\
																& S'			\ar[r]			& S.
  \end{tikzcd}
 \]
 But this means that $(x',v',\phi')$ is not a Riemann-Zariski point, a contradiction.

 Assume now that $X'$ is quasi-compact and all residue field extensions of $X' \to X$ are algebraic,
  and that $\Spa(X',S') \to \Spa(X,S)$ maps Riemann-Zariski points to Riemann-Zariski points.
 We have to show that for any diagram
 \[
  \begin{tikzcd}
   \Spec k		\ar[r]	\ar[d]			& X'			\ar[d]			\\
   \Spec \O	\ar[r]	\ar[ur,dashed]	& X \times_S S',
  \end{tikzcd}
 \]
 where~$\O$ is a valuation ring of the field~$k$, the dotted arrow exists.

 Without loss of generality we may assume that $k$ equals the residue field of the image of $\Spec k \to X'$.
 The solid diagram then defines a point $(x',v',\phi')$ in $\Spa(X',S')$
  whose image $(x,v,\phi)$ in $\Spa(X,S)$ has center in~$X$.
 Applying \cref{RZ_specialization} we obtain a horizontal specialization $(y',w',\psi')$ of $(x',v',\phi')$ which is a Riemann-Zariski point.

 By assumption it maps to a Riemann-Zariski point $(y,w,\psi)$ in $\Spa(X,S)$.
 We know that $(y,w,\psi)$ is a horizontal specialization of $(x,v,\phi)$ and that $(x,v,\phi)$ has center in~$X$.
 Therefore, also $(y,w,\psi)$ has center in~$X$.
 Being a Riemann-Zariski point with center in~$X$, $(y,w,\psi)$ has to correspond to a trivial valuation.
 By our hypothesis on residue field extensions this implies that $(y',w',\psi')$ corresponds to a trivial valuation, as well.
 In particular, $(x',v',\phi')$ has center in~$X'$ which is moreover compatible with the given center of $(x,v,\phi)$ in~$X$.
 In other words, the dotted arrow exists.
\end{proof}

\begin{lemma}\label{uc-prop}
 \begin{enumerate}[\rm (i)]
   \item (Composition) If the squares
    \[
  \begin{tikzcd}
   X''	\ar[r]	\ar[d]	& X'	\ar[d]	&	& X'	\ar[r]	\ar[d]	& X	\ar[d]	\\
   S''	\ar[r]			& S'			&	& S'	\ar[r]			& S
  \end{tikzcd}
  \]
  have universally closed diagonal, then the same is true for the square
  \[
  \begin{tikzcd}
   X''	\ar[r]	\ar[d]	& X	\ar[d]	\\
   S''	\ar[r] & S.
  \end{tikzcd}
  \]
   \item (Base change) Consider two squares
   \[
  \begin{tikzcd}
   Y	\ar[r]	\ar[d]	& X	\ar[d]	&	& X	\ar[d]	& X'	\ar[l]	\ar[d]	\\
   T	\ar[r]			& S			&	& S			& S'	\ar[l]
  \end{tikzcd}
 \]
 and assume that the right hand one has universally closed diagonal. Then the same is true for the square
   \[
  \begin{tikzcd}
Y	\ar[d]&   Y'=Y	\times_XX' \ar[l] \ar[d] \\
   T &T'=T\times_S S'	\ar[l]
  \end{tikzcd}
 \]
 \item (Limits) Let $X\to S$ be a scheme morphism and let $(X_i\to S_i)_{i\in I}$ be a family of morphisms over $X\to S$ such that all $X_i$ are quasi-compact and all squares
 \[
 \begin{tikzcd}
   X_i\ar[r] \ar[d]&X\ar[d]\\
   S_i\ar[r]& S
 \end{tikzcd}
 \]
 have universally closed diagonal. Then the square
 \[
 \begin{tikzcd}
   \lim X_i\ar[r] \ar[d]&X\ar[d]\\
   \lim S_i\ar[r]& S
 \end{tikzcd}
 \]
 has universally closed diagonal.
 \end{enumerate}
\end{lemma}

\begin{proof}
We leave the standard verifications of (i) and (ii) to the reader. For (iii) note that  by the valuative criterion for universal closedness \cite[TAG 01KF]{stacks-project}, it suffices to show that in every solid diagram
 \[
  \begin{tikzcd}
   \Spec K		\ar[r]	\ar[d]			& \lim_{i \in I} X_i				\ar[d]	\\
   \Spec \O	\ar[r]	\ar[ur,dashed]	& X \times_{S} \lim_{i \in I}S_i,
  \end{tikzcd} \eqno (*)
 \]
 where~$\O$ is a valuation ring with quotient field~$K$, the dotted arrow exists.
 For every $i \in I$ denote by~$M_i$ the set of all possible dotted arrows in
 \[
  \begin{tikzcd}
   \Spec K		\ar[r]	\ar[d]			& X_i					\ar[d]	\\
   \Spec \O	\ar[r]	\ar[ur,dashed]	& X\times_{S} S_i.
  \end{tikzcd}
 \]
 By assumption~$M_i$ is non-empty.
 Since~$X_i$ is quasi-compact, it is furthermore finite.
 The reason is that given a finite affine cover $X_i = \bigcup_k U_k$ every dotted arrow factors through some affine open~$U_k$
  and for every~$U_k$ there is at most one dotted arrow.
 The~$M_i$ thus form an inverse system of finite sets.
 Hence, their limit $M = \lim_i M_i$ is non-empty and every element of~$M$ provides a dotted arrow for the diagram~($*$).
\end{proof}

\begin{lemma} \label{cover_RZ_affinoids}
 Let $S$ be quasi-compact and quasi-separated and  $X \to S$  a separated morphism of finite type.
 Then there is an open covering of $\Spa(X,S)$ by finitely many affinoids $\Spa(A_i,A_i^+)$ coming from squares
 \[
  \begin{tikzcd}
   \Spec A_i	\ar[r]	\ar[d]	& X	\ar[d]	\\
   \Spec A_i^+	\ar[r]			& S
  \end{tikzcd}
 \]
 with universally closed diagonal such that $\Spec A_i \to X$ is a quasi-compact open immersion
  and $\Spec A_i^+ \to S$ is of finite type.
 In particular, $\Spa(A_i,A_i^+) \to \Spa(X,S)$ is a quasi-compact open immersion which is a Riemann-Zariski morphism.
\end{lemma}

\begin{proof}
By Deligne's generalization of Nagata's compactification theorem \cite[Theorem~4.1]{Con07}, we can find a compactification~$\bar{X}$ of~$X$ over~$S$. Blowing up the complement of~$X$ in~$\bar{X}$ if necessary, we may assume that $\bar{X}\smallsetminus X$ is the support of a Cartier divisor.  Hence, we can find a covering of~$\bar{X}$ by finitely many affine open subschemes~$\bar U_i$ such that  $U_i= \bar U_i \cap X$ is affine as well.  Then the affinoids $\Spa(U_i,\bar U_i)$ form an open covering of $\Spa (X,S)$. Moreover, since $\bar X\to S$ is separated, $X\hookrightarrow \bar X \times_S X$ is a closed immersion and applying $\bar U_i \times_{\bar{X}} - $, we see that also $U_i \to \bar{U}_i \times_{S} X $ is a closed immersion for all $i$. Hence the squares
\[
\begin{tikzcd}
U_i\ar[r]\ar[d]&X\ar[d]\\
\bar U_i \ar[r]&S
\end{tikzcd}
\]
have a universally closed diagonal.
\end{proof}

\begin{lemma} \label{dominate_RZ}
 Let $f:\Spa(X',S') \to \Spa(X,S)$ come from a square
 \[
  \begin{tikzcd}
   X'	\ar[r]	\ar[d]	& X	\ar[d]	\\
   S'	\ar[r]			& S
  \end{tikzcd}
 \]
 of separated, finite type morphisms of quasi-compact and quasi-separated schemes. Assume that $f$ is \'{e}tale and surjective. Then there exists a finite set of morphisms of finite type $T_i \to S'$  such that the induced morphism
  \[
 \coprod_i \Spa (X',T_i) \longrightarrow \Spa (X,S)
 \]
is a surjective (\'{e}tale) Riemann-Zariski morphism.
\end{lemma}
\begin{proof} Let $\bar{X}$ be a compactification of~$X$ over~$S$ (see \cite{Con07}, Theorem~4.1).
 Then~$f$ coincides with the induced morphism
 \[
  \Spa(X',\bar{X} \times_X S') \to \Spa(X,\bar{X}).
 \]
 We may thus assume that $X \to S$ is an open immersion.
 By the same arguments we reduce to the case where also $X' \to S'$ is an open immersion.
 We thus have a diagram
 \[
  \begin{tikzcd}
   X'	\ar[drr]	\ar[dr]	\ar[ddr,open]	&									&				\\
															& X \times_S S'	\ar[r]	\ar[d,open]	& X	\ar[d,open]	\\
															& S'			\ar[r]				& S.
  \end{tikzcd}
 \]
 The map $X' \to X \times_S S'$ is thus an open immersion as well.
 We denote by $Z \subset X \times_S S'$ the complement of~$X'$ (with reduced scheme structure).
 Let $I$ be a set indexing the isomorphism classes of $X \times_S S'$-modifications of $S'$,
  i.e., for each $i \in I$ we have a diagram
 \[
  \begin{tikzcd}
   X \times_S S'	\ar[rr,open]	\ar[dr,open]	&	& S_i	\ar[dl,"\text{proper, birational}"]	\\
													& S'.
  \end{tikzcd}
 \]
 Denote by~$Z_i$ the closure of~$Z$ in~$S_i$ and set $T_i = S_i  \smallsetminus Z_i$.
 We obtain a diagram
 \[
  \begin{tikzcd}
   X'	\ar[drrr]	\ar[dr,open]	\ar[ddr,open]	&									&				\\
																	& X \times_S T_i	\ar[r,open]	\ar[d,open]	& X \times_S S_i	\ar[r]	\ar[d,open]	& X	\ar[d,open]	\\
																	& T_i				\ar[r,open]				& S_i			\ar[r]				& S.
  \end{tikzcd}
 \]
 In fact we have
 \[
  X \times_S T_i = T_i \cap (X \times_S S_i) = (X \times_S S')  \smallsetminus Z = X',
 \]
where we have used that $X \times_S S_i = X \times_S S'$. Hence, by \cref{RZ_closed_diagonal}, $\Spa(X',T_i) \to \Spa(X,S)$ is Riemann-Zariski. Furthermore, it dominates $\Spa(X',S') \to \Spa(X,S)$.

 What is left to show is that the family $(\Spa(X',T_i) \to \Spa(X,S))_{i \in I}$ is surjective.
 Then we know by the quasi-compactness of $\Spa(X,S)$ that finitely many $\Spa(X',T_i)$ cover $\Spa(X,S)$
  and we can take their disjoint union.
 In order to show this surjectivity it suffices to prove that every \emph{closed} point of $\Spa(X,S)$ is contained in the image of one of the $\Spa(X',T_i)$.
 Let $(x,v,\phi)$ be a closed point of $\Spa(X,S)$ and pick a closed point $(x',v',\phi')$ in its preimage in $\Spa(X',S')$.

 We can view $(x',v',\phi')$ as a point of $\Spa(X \times_S S',S')$ and as such it is a Riemann-Zariski point for the following reason.
 By \cref{characterize_RZ_points} we have to show that~$x'$ is a closed point in
 \[
  \Spec \O_{v'} \times_{S'} (X \times_S S') = \Spec \O_{v'} \times_S X.
 \]
 But this is equivalent to saying that
 \[
  \begin{tikzcd}
   \Spec k(x')		\ar[r]	\ar[d]	& X	\ar[d,open]	\\
   \Spec \O_{v'}	\ar[r]			& S
  \end{tikzcd}
 \]
 has universally closed diagonal, which in turn is equivalent by \cref{RZ_closed_diagonal} to
 \[
  \Spa(k(x'),\O_{v'}) \to \Spa(X,S)
 \]
 being Riemann-Zariski.
 This is true as the image $(x,v,\phi)$ of $(x',v',\phi')$ in $\Spa(X,S)$ is a closed point.

 For any $i \in I$ we have $\Spa(X \times_S S',S') \cong \Spa(X \times_S S',S_i)$.
 The point $(x',v',\phi')$ in $\Spa(X \times_S S',S')$ identifies with the point $(x',v',\phi_i)$,
  where $\phi_i : \Spec \O_{v'} \to S_i$ is the unique lift of the morphism $\phi' : \Spec \O_{v'} \to S'$ coming from the valuative criterion of properness.
 Denote by~$c_i$ the center of $(x',v',\phi_i)$ in~$S_i$, i.e., the image of the closed point of $\Spec \O_{v'}$ under $\phi_i$.
 Now \cite{TemTyom18}, Lemma~5.1.3 gives us $i \in I$ such that $c_i$ is not contained in~$Z_i$.
 In other words, $(x',v',\phi_i) \in \Spa(X',T_i)$
\end{proof}

\section{\v{C}ech cohomology of discretely ringed adic spaces}
In this section we transfer the results of \cref{sec:cech} to the realm of adic spaces. Let $\tau$ denote one of the topologies `$\et$' (\'{e}tale), `$\set$' (strongly \'{e}tale) or `$t$' (tame). We make the general assumption  that all schemes and scheme morphisms are quasi-compact and quasi-separated.

\begin{proposition}\label{acyclic adic}
 For every affinoid adic space $\Spa(A,A^+)$ there is a surjective Riemann-Zariski pro-$\tau$ morphism
 \[
  \Spa(\widetilde{A},\widetilde{A}^+) \to \Spa(A,A^+)
 \]
 with $(\widetilde{A},\widetilde{A}^+)$ $\tau$-acyclic.
\end{proposition}

\begin{proof} As in the proof of \cref{acyclic tame}, we follow the method of Bhatt-Scholze, \cite[Proof of Lemma 2.2.7]{BS15}.
 Let~$I$ be the set of isomorphism classes of surjective Riemann-Zariski $\tau$-morphisms $\Spa(A',A'^+) \to \Spa(A,A^+)$.
 For every $i \in I$ pick a representative $\Spa(B_i,B_i^+) \to \Spa(A,A^+)$ and set
 \[
  (A_1,A_1^+) := \colim_{J \subset I \text{ finite}} \bigotimes_{j \in J} (B_j,B_j^+),
 \]
 where the tensor product is taken over $(A,A^+)$ and the (filtered) colimit is indexed by the poset of finite subsets of $I$.
 There is an obvious ind-$\tau$ map $(A,A^+) \to (A_1,A_1^+)$ which induces a surjective Riemann-Zariski morphism on adic spectra.
 Moreover, by \cref{dominate_RZ}, any surjective $\tau$- morphism $\Spa(B,B^+) \to \Spa(A,A^+)$ is dominated by a Riemann-Zariski surjective $\tau$-morphism $\Spa(X,S) \to \Spa(A,A^+)$.
 By \cref{cover_RZ_affinoids} we may assume that $\Spa(X,S)$ is affinoid, i.e., $\Spa(X,S) \cong \Spa(A_i,A_i^+)$ for some $i \in I$.
 We therefore obtain an $(A,A^+)$-homomorphism $(B,B^+) \to (A_1,A_1^+)$.
 This implies that $(A,A^+) \to (B,B^+)$ splits after base change to $(A_1,A_1^+)$.
 Iterating the construction with $(A_1,A_1^+)$ replacing $(A,A^+)$ and proceeding inductively defines a tower
 \[
  (A,A^+) \to (A_1,A_1^+) \to (A_2,A_2^+) \to \cdots
 \]
 of faithfully flat Riemann-Zariski $(A,A^+)$-algebras with ind-$\tau$ transition maps.
 Set $(\widetilde{A},\widetilde{A}^+) = \colim (A_n,A_n^+)$.
 As \'{e}tale homomorphisms are finitely presented, one checks that any faithfully flat $\tau$-$(\widetilde{A},\widetilde{A}^+)$-algebra has a section, so $(\widetilde{A},\widetilde{A}^+)$ is $\tau$-acyclic.

 Being a filtered colimit of faithfully flat Riemann-Zariski $\tau$-$(A,A^+)$-algebras, $(A,A^+) \to (\widetilde{A},\widetilde{A}^+)$ is ind-$\tau$ and Riemann-Zariski.
 Moreover, since the inverse limit of finite non-empty sets is non-empty, every point of $\Spa(A,A^+)$ has a preimage in
 $ \Spa(\widetilde{A},\widetilde{A}^+)$.
\end{proof}

\begin{theorem} \label{product_acyclic adic}
 Let $\cX$ be a discretely ringed adic space with the property that any finite subset of~$\cX$ is contained in an affinoid open. Let $\tau \in \{\et, \set, t\}$.
 If~$\tau=t$  assume that~$\cX$ is of pure characteristic $p \ge 0$.
 Suppose we are given for $i= 1,\ldots,n$ pro-$\tau$ Riemann-Zariski morphisms $\cU_i=\Spa(A_i,A_i^+) \to \cX$ with $(A_i,A_i^+)$ $\tau$-acyclic.
 Then the fiber product
 \[
  \cU:=\Spa(A_1,A_1^+) \times_\cX \Spa(A_2,A_2^+) \times_\cX \ldots \times_\cX \Spa(A_n,A_n^+)
 \]
 is affinoid and $\tau$-acyclic.
\end{theorem}

 \begin{proof}
 If $f: \cY \to \cX$ is a Riemann-Zariski morphism which factors through an open subspace $\cZ \subset \cX$, then $\cY\to \cZ$ is also Riemann-Zariski. Indeed, if $x\in \cZ$ is Riemann-Zariski in $\cX$, it is Riemann-Zariski in $\cZ$.

 Moreover, a Riemann-Zariski morphism sends closed points to Riemann-Zariski points. Hence, if there is an affinoid open subspace $\cV\subset \cX$ such that all $\cU_i$ map to~$\cV$, the result follows from \cref{product of strongly acyclic Huber pairs}.

In the general case, let $p_i\in \cU_i$, $i=1,\ldots,n$, be closed points. By assumption, there is an affinoid open $\cV =\Spa(B,B^+)\subset \cX$ containing the images of $p_1,\ldots,p_n$. Since $\cU_i$ is quasi-acyclic, every component of $\cU_i$ has a unique closed point and it maps to $\cV$ if and only if its closed point maps to $\cV$. By \cref{topology}, $\varphi_i: \cU_i^c \to  \pi_0(\cU_i)$ is a homeomorphism of profinite spaces. Therefore, we find a closed and open subset $\cW_i \subset \cU_i^c$ which contains $p_i$ and maps to $\cV$. Then also the preimage $\cV_i\subset  \cU_i$ of $\varphi_i(\cW_i)\subset \pi_0(\cU_i)$ under $\cU_i\to \pi_0(\cU_i)$ maps to $\cV$. Hence $\cV_i$ is a closed and open subset of $\cU_i$ containing $p_i$ and mapping to $\cV$. Being closed and open in $\cU_i$, $\cV_i$ is the adic spectrum of an ind-\'{e}tale, $\tau$-acyclic $(B,B^+)$-algebra.
Moreover $\cV_i \to \cV$ is a Riemann-Zariski morphism for the following reason.
The composition $\cV_i \to \cU_i \to \cX$ is Riemann-Zariski because the first morphism is closed and the second is Riemann-Zariski by assumption.
Hence $\cV_i \to \cV$ is Riemann-Zariski by the remark at the beginning of the proof.

By the first part of the proof, $\cV(p_1,\ldots,p_n):=\cV_1\times_\cX\cdots \times_\cX \cV_n$ is a closed and open subspace of $\cU$, hence affinoid ind-\'{e}tale and $\tau$-acyclic. Any nonempty closed subset in a spectral space contains a closed point.
Hence, varying the points $p_1,\ldots, p_n$, the $\cV(p_1,\ldots,p_n)$ cover $\cU$. Since $\cU$ is quasi-compact, we find a finite subcovering. Replacing the $\cV(p_1,\ldots,p_n)$ by closed and open subspaces, which are ind-\'{e}tale and $\tau$-acyclic, we may assume that the finite union is disjoint. Hence $\cU$ is affinoid and $\tau$-acyclic.
\end{proof}

Using \cref{acyclic adic} and \cref{product_acyclic adic} instead of \cref{acyclic tame} and \cref{product of tamely acyclic}, the proof of \cref{cechcompare adic} below is word-by-word the same as the proof of \cref{cechcompare}.

\begin{theorem}[Comparison with \v{C}ech cohomology] \label{cechcompare adic} Let $\cX$ be a quasi-compact discretely ringed adic space  with the property that any finite subset of~$\cX$ is contained in an affinoid open. Let $\tau \in \{\et, \set, t\}$.
 If\/ $\tau=t$  assume that~$\cX$ is of pure characteristic $p \ge 0$.

Then, for every presheaf $P$ of abelian groups on $\cX_\tau$ with sheafification $P^{\# \tau}$, the natural maps
\[
\check{H}^n_\tau (\cX,P) \to H^n_\tau(\cX, P^{\#\tau})
\]
are isomorphisms for all $n\ge 0$.
\end{theorem}

\section{Comparison between algebraic and adic tame cohomology}

For $X/S$ we consider the tame site $(X/S)_t$ and the adic tame site $\Spa(X,S)_t$ and want to compare their cohomology. Since there is no obvious site morphism between the two, we consider an additional site $\Spa(X,S)_{t,\Nis}$ which admits morphisms
\[
\begin{tikzcd}
\Spa(X,S)_t & \Spa(X,S)_{t,\Nis}\arrow[r, "\phi"] \arrow[l, "\psi"'] & (X/S)_t
\end{tikzcd}
\]
and we will prove that there are isomorphisms for all $q\ge 0$:
\renewcommand{\arraystretch}{1.5}
\[
\begin{array}{rcl}
H^q_t(X/S,F) & \cong& H^q_{t,\Nis}(\Spa(X,S),\phi^*F),\quad  F\in \Sh_t(X/S)\\
H^q_{t}(\Spa(X,S),\psi_*G) &\cong &H^q_{t,\Nis}(\Spa(X,S),G),\quad  G\in \Sh_{t,\Nis}(\Spa(X,S)).
\end{array}
\]

\begin{definition} Let $\cX$ be a discretely ringed adic space. The \emph{Nisnevich-tame site} $\cX_{\Nis,t}$ is defined by the following data:

\smallskip
The category $\Cat(\cX_{\Nis,t})$ is the category of \'{e}tale morphisms of adic spaces $\cU\to \cX$.

A family $(p_i: \cU_i \to \cU)$ of morphisms in $\Cat (\cX_{\Nis,t})$ is a covering if for every point $u\in \cU$ there exists an index~$i$ and a point $u_i\in \cU_i$ mapping to $u$ such that $p_i$ is tame at $u_i$.
\end{definition}

The category $\Cat (\cX_t)$ is a full subcategory of $\Cat (\cX_{\Nis,t})$ and obviously every covering in $\cX_t$ is also a covering in $\cX_{\Nis,t}$. This explains the existence of the morphism $\psi: \cX_{\Nis,t}\to \cX_{t}$ (in particular for $\cX=\Spa(X,S)$).
The morphism $\phi$ is defined by mapping an \'{e}tale $X$-scheme $U$ to  $\Spa(U,S)$, which lies \'{e}tale over $\Spa(X,S)$.

\medskip
The comparison result for $\psi$ will follow from  the openness of the tame locus (\cref{opennness-tame-locus}).

\begin{lemma} \label{dominate Nis covers}
 Let $\cX$ be a discretely ringed adic space and $\cU$ an object in $\cX_t$.
 The coverings of $\cU$ in $\cX_t$ are cofinal among the coverings in $\cX_{t,\Nis}$.
\end{lemma}

\begin{proof}
 Consider a covering $(\cU_i \to \cU)$ in $\cX_{t,\Nis}$.
 For every~$i$ we denote by $\cV_i \subseteq \cU_i$ the subset of points where $\cU_i \to \cU$ is tame.
 This is an open subset by {\cite[Corollary 4.4]{HueAd}}.
 Moreover, $(\cV_i \to \cU)$ is a surjective family by the definition of coverings in $\cX_{t,\Nis}$.
 Hence, $(\cV_i \to \cU)$ is a covering dominating $(\cU_i \to \cU)$.
\end{proof}

\begin{lemma} \label{psi isomorphism cohomology}
 The morphism of sites
 \[
  \psi: \cX_{t,\Nis} \to \cX_t
 \]
 induces isomorphisms
 \[
  H^q_{t}(\cX,\psi_*G) \cong H^q_{t,\Nis}(\cX,G)
 \]
 for all $q \ge 0$ and for every $G \in \Sh_{t,\Nis}(\cX)$.
\end{lemma}

\begin{proof}
 We have to show that the higher direct images $R^q\psi_* G$ vanish for $q > 0$.
 Unravelling the definitions this comes down to showing that every covering $(\cV_j \to \cU)$ in $\cX_{t,\Nis}$ of an object $\cU$ in $\cX_t$ is dominated by a covering in $\cX_t$.
 This is the assertion of \cref{dominate Nis covers}.
\end{proof}

\begin{lemma} \label{compare two cech}
 Let $\cX$ be a discretely ringed adic space and $F$ a presheaf on $\cX_{t,\Nis}$.
 Then for any $\cU$ in $\cX_t$ the natural homomorphism
 \[
  \check{H}^q_t(\cU,\psi_*F) \to \check{H}^q_{t,\Nis}(\cU,F)
 \]
 is an isomorphism for all $q \ge 0$.
\end{lemma}

\begin{proof}
 This is a direct consequence of \cref{dominate Nis covers}.
\end{proof}

\begin{proposition} Let $\cX$ be a quasi-compact discretely ringed adic space  with the property that any finite subset of~$\cX$ is contained in an affinoid open. Moreover assume that~$\cX$ is of pure characteristic $p \ge 0$.  Let $F$ be a presheaf on $\cX_{t,\Nis}$ with sheafification $F^{\#}$
 Then there are natural isomorphisms
 \[
  \check{H}^q_{t,\Nis}(\cX,F) \cong H^q_{t,\Nis}(\cX,F^{\#})
 \]
 for all $q \ge 0$.
\end{proposition}

\begin{proof}
 Consider the commutative diagram
 \[
  \begin{tikzcd}
   \check{H}^q_t(\cX,\psi_*F)	\ar[r]	\ar[dd]	& H^q_t(\cX,(\psi_*F)^{\#})	\ar[d]	\\
						& H^q_t(\cX,\psi_*(F^{\#}))	\ar[d]	\\
   \check{H}^q_{t,\Nis}(\cX,F)	\ar[r]		& H^q_{t,\Nis}(\cX,F^{\#}).
  \end{tikzcd}
 \]
 The upper horizontal map is an isomorphism by \cref{cechcompare adic}, the left vertical map by \cref{compare two cech}, and the lower right vertical map by \cref{psi isomorphism cohomology}.
 In order to see that the upper right vertical map is an isomorphism we note that sheafification is given by applying $\check{H}^0$ twice.
 Hence, $(\psi_*F)^{\#} \to \psi_*(F^{\#})$ is an isomorphism by \cref{compare two cech} for $q = 0$.
\end{proof}

\begin{lemma}\label{cech-adic-algebraic}
 Let $X \to S$ be a morphism of  schemes and $F\in \PrSh_t(X/S)$ a presheaf.
We consider the morphism of sites
 \[
  \phi: \Spa(X,S)_{t,\Nis} \to (X/S)_t
 \]
 and let $G=\phi^*_{\mathit{pre}}(F) \in \PrSh_{t,\Nis}(\Spa(X,S))$ be the presheaf pull-back of $F$.
 Then the natural homomorphism on \v{C}ech cohomology
 \[
\check{H}^n_t(X/S, F) \lang \check{H}^n_{t, \Nis}(\Spa(X,S), G)
\]
is an isomorphism for all $n\ge 0$.
\end{lemma}

\begin{proof} We consider all diagrams
 \[
  \begin{tikzcd}
   U	\ar[r]	\ar[d]	& X	\ar[d]	\\
   T	\ar[r]			& S,
  \end{tikzcd}
 \]
 with $U \to X$ \'{e}tale and $T \to S$ of finite type, i.e., $\Spa(U,T) \in \Spa(X,S)_{t,\Nis}$.
 By definition of $\phi^*_{\mathit{pre}}$, we have $G(\Spa(U,T))=\colim_V F(V)$,
  where the colimit is taken over all diagrams
 \[
  \begin{tikzcd}
   U	\ar[r]	\ar[d]	& V	\ar[r]			\ar[d]	& X	\ar[d]	\\
   T	\ar[r]			& S	\ar[r,equal]			& S
  \end{tikzcd}
 \]
 with $V \to X$ \'etale. Since $U$ is a final object, we obtain $G(\Spa(U,T)) = F(U)$.

 For a tame covering $\cU=(U_\alpha \to X)$ in $(X/S)_t$ we consider the pull-back covering $\phi^*\cU=(\Spa(U_\alpha,S) \to \Spa(X,S))$ in $\Spa(X,S)_{t,\Nis}$. We have a natural isomorphism of \v{C}ech complexes $\check{C}^\bullet(\cU, F) \cong \check{C}^\bullet(\phi^*\cU, G)$. It therefore suffices to show that the colimit of the \v{C}ech cohomology groups of $G$ with respect to all pull-back coverings coincides with the colimit with respect to all coverings of $\Spa(X,S)$ in the $(t,\Nis)$-topology. Among these, coverings of the form $\cV=(\Spa(U_\alpha,T_\alpha) \to \Spa(X,S))$ with all $U_\alpha \to X$ \'{e}tale and all $T_\alpha \to S$ of finite type are cofinal, so we can restrict to coverings of this type. The observation at the beginning of this proof shows that the \v{C}ech cohomology of $G$ for $\cV$ coincides with the \v{C}ech cohomology of $F$ for the covering $(U_\alpha \to X)$ and hence with the \v{C}ech cohomology of $G$ for the pull-back covering $r(\cV):=(\Spa(U_\alpha,S) \to \Spa(X,S))$, of which $\cV$ is a refinement. Hence our result follows from \cref{colimlemma} below with $J$ the subset of pull-back coverings and $H$ the \v{C}ech cohomology for $G$.
\end{proof}

\begin{lemma}\label{colimlemma}
Let $I$ be a partially ordered directed set and $J\subset I$ a directed subset. Let $(H_i)_{i\in I}$ be a system of abelian groups indexed by $I$. Assume there is an order preserving retraction $r$ to the inclusion $J\subset I$ such that $r(i)\le i$ for all $i\in I$ and the natural map $H_{r(i)}\to H_i$ is an isomorphism for all $i\in I$. Then the natural map
\[
\colim_J H_j \to \colim_I H_i
\]
is a isomorphism.
\end{lemma}

\begin{proof}
Surjectivity follows from the surjectivity of $H_{r(i)}\to H_i$ for all $i\in I$. Let $j\in J$ and $a\in H_j$ with trivial image in $H_i$ for some $i\in I$, $i\ge j$. We have to find a $j'\in J$, $j'\ge j$, such that $a$ has trivial image in $H_{j'}$. By our assumptions, we have $j=r(j)\le r(i)\le i$ and $H_{r(i)}\to H_i$ is injective. Hence $j'=r(i)$ has the required property.
\end{proof}

Now we are ready to prove our comparison theorem.
\begin{theorem}\label{main compare}
Let $S$ be an affine $\F_p$-scheme for some prime number $p$ and let $X$ be an $S$-scheme.  Then for any sheaf $F\in \Sh_t(X/S)$ the homomorphism
\[
H^n_t(X/S, F) \lang H^n_{t, \Nis}(\Spa(X,S), \phi^*F)
\]
is an isomorphism for all $n\ge 0$.
\end{theorem}

\begin{proof} Let $X=\bigcup U_i$ be an affine open covering.  Then  $\Spa(X,S)=\bigcup \Spa(U_i,S)$ is also an open covering.  By the \v{C}ech-to-derived spectral sequences on both sides, we may therefore assume that $X$ itself is  affine. Let $G=\phi^*_{\mathit{pre}}F$, i.e., $\phi^*F=G^\#$. Then the statement follows from the commutative diagram
\[
\begin{tikzcd}
\check{H}^n_t(X/S, F)\rar{a}\dar{b}&\check{H}^n_{t, \Nis}(\Spa(X,S), G)\dar{c}\\
H^n_t(X/S, F)\rar&H^n_{t, \Nis}(\Spa(X,S), \phi^*F),
\end{tikzcd}
\]
in which $a$, $b$ and $c$ are isomorphisms by \cref{cech-adic-algebraic}, \cref{cechcompare} and \cref{cechcompare adic}, respectively.
\end{proof}

\section{Purity and homotopy invariance}

With the help of our comparison results between algebraic and adic tame cohomology, we can now exploit the results of \cite{HueAd}.

The version of resolution of singularities used in this paper is the following.
\begin{definition} \label{ros} Let~$S$ be a noetherian scheme.
 We say that \emph{resolution of singularities holds over~$S$} if for any reduced scheme~$X$ of finite type over~$S$ there is a locally projective birational morphism $X' \to X$ such that~$X'$ is regular and $X' \to X$ is an isomorphism over the regular locus of~$X$.  (By \cite[IV, 7.9.5]{EGAIV.2}, this particularly implies that $S$ is quasi-excellent.)
\end{definition}

A sheaf $F\in \Sh_t(X/S)$ is \emph{locally constant} if there is a tame covering $(U_i \to X)_i$ such that the restriction $F|_{U_i}$ is isomorphic to a constant sheaf $\underline{A_i}$ on $U_i$ for all $i$.

\begin{theorem}[Purity] \label{purity2}
 Let~$S$ be an affine noetherian scheme of characteristic $p > 0$ and $X$ a regular scheme which is separated and essentially of finite type over~$S$.
 Assume that resolution of singularities holds over~$S$.
 Then for any open dense subscheme $U \subset X$  and every locally constant $p$-torsion sheaf $F\in \Sh_t(X/S)$ the natural map
 \[
 H^q_t(X/S, F) \lang H^q_t(U/S,F|_U)
 \]
is an isomorphism for all $q\ge 0$.
\end{theorem}

\begin{proof} Let $(U_i \to X)_i$ be a trivializing tame covering for $F$. Then we have compatible \v{C}ech-to-derived spectral sequences
\[
E_2^{rs}=\check{H}^r((U_i\to X)_i, \cH^s_t(-/S,F)) \Rightarrow H^{r+s}_t(X/S, F),\ \text{and}
\]
\[
E_2^{rs}=\check{H}^r((U_i\times_X U\to U)_i, \cH^s_t(-/S,F|_U)) \Rightarrow H^{r+s}_t(U/S, F|_U).
\]
We therefore may assume that $F$ is constant. By \cref{colimitsheaves}, we may assume that $F=\underline{A}$ for a finite abelian group $A$. For this it suffices to consider the cyclic case $A=\Z/p^n\Z$ and by d\'{e}vissage , we may assume $n=1$. Then the statement follows from its adic version \cite[Corollary 14.5]{HueAd}, and the comparison results \cref{main compare} and \cref{psi isomorphism cohomology}.
\end{proof}

\begin{remarks}
1. Let us say a few words about why \cref{purity2} is a purity statement.
Let $\nu_n(r)$ be the logarithmic deRham Witt sheaves as defined in \cite{Milne86}.
For a regular closed immersion $Z \to X$ of codimension~$c$, we expect to hold the purity statement
\[
 H^q_{t,Z}(X/S,\nu_n(r)) \cong H^{q-c}_t(Z/S,\nu_n(r-c)).
\]
Since $\nu_n(0) \cong \Z/p^n\Z$ and $\nu_n(r) = 0$ for $r < 0$, this implies $H^q_{t,Z}(X/S,\Z/p^n\Z)=0$ for all $q$, which is essentially equivalent to \cref{purity2} for constant coefficients.

\smallskip\noindent
2. Being an \'{e}tale sheaf over a tame covering, a locally constant tame sheaf is already an \'{e}tale sheaf. In view of the comparison result \cref{compare invertible}, purity for the tame cohomology of locally constant \emph{prime-to-$p$} torsion sheaves is identical to  purity for \'{e}tale cohomology, see \cite{Fuj00}.
\end{remarks}

\begin{theorem}[Homotopy invariance] \label{homotopyinvariance}
 Let~$S$ be an affine noetherian scheme of characteristic $p > 0$ and $X$ a regular scheme which is essentially of finite type over~$S$.
 Assume that resolution of singularities holds over~$S$.
 Then for every locally constant torsion sheaf $F\in \Sh_t(X/S)$ the natural map
 \[
   H^q_t(X/S,F) \lang H^q_t(\A^1_X/S,\pr^*F),
 \]
where $\pr: \A^1_X \to X$ is the natural projection, is an isomorphism for all $q\ge 0$.
\end{theorem}

\begin{proof} As in the proof of \ref{purity2}, we may reduce to the case where  $F=\underline{\Z/q\Z}$ for a prime number $q$. If $q\ne p$, then the result follows from \cref{compare invertible} and the homotopy invariance of \'{e}tale cohomology with invertible coefficients \cite[VI,\ Corollary 4.20]{Milne80}. For $q=p$ the result follows from its adic version  \cite[Corollary 14.6]{HueAd}, and the comparison results \cref{main compare} and \cref{psi isomorphism cohomology}.
\end{proof}

\section{Connection to Suslin homology}\label{sec_suslin}
Let $k$ be an algebraically closed field and assume that resolution of singularities holds over~$k$ (see \cref{ros}).
Let $X/k$ a connected scheme of finite type and  $m\ge 1$ an integer.
In \cite{SV96}, A.~Suslin and V.~Voevodsky defined the Suslin (or singular) homology groups $H_n^S(X,\Z/m\Z)$ and constructed for $(m,\ch(k))=1$ an isomorphism between \'{e}tale and Suslin cohomology
$$H^n_\et(X,\Z/m\Z) \cong H^n_S(X,\Z/m\Z)=\Hom(H^S_n(X,\Z/m\Z), \Z/m\Z).$$
Originally, \cite{SV96} had to assume resolution of singularities but later this could be avoided by using deJong's alterations.

In \cite{GS16}, T.~Geisser and the second author extended this result in degree~1 to general $m$ not necessarily prime to $\ch(k)$ by replacing the group $H^1_\et(X,\Z/m\Z)$ by the tame cohomology group $H^1_t(X,\Z/m\Z)$, which was defined in an ad hoc man\-ner as the dual of the curve-tame fundamental group (cf.\ \cite{KeSch10}), i.e.,
\[
H^1_t(X,\Z/m\Z):= \Hom_\mathrm{cts}(\pi_1^{ct}(X), \Z/m\Z).
\]
Here the assumption on resolution of singularities could not be avoided. A generalization of this result to higher degrees was not possible since a definition of (higher) tame cohomology groups was missing.

\medskip
Having tame cohomology at hand now, we construct for all $n\ge 0$ a natural pairing
\[
H_n^S(X,\Z/m\Z) \times H^n_t(X/k,\Z/m\Z) \lang \Z/m\Z,
\]
which defines maps
\[
H^n_t(X/k,\Z/m\Z) \lang H^n_S(X,\Z/m\Z)
\]
in all degrees $n$. The construction is the following.

\bigskip\noindent
Let $\Sch/k$ be the category of separated schemes of finite type over $k$ and
$\Cor/k$  the category with the same objects as $\Sch/k$ and finite correspondences as morphisms. Recall that the group $\Cor_k(X,Y)$ of \emph{finite correspondences} from $X$ to $Y$ is the group of relative cycles in $X\times_kY/X$ which are finite, equidimensional and universally integral, see \cite[A.1]{MVW06} or \cite[\S9]{CiDe12}. The graph functor
\[
\Sch/k \lang \Cor/k, \quad f\in \Mor_k(X,Y) \longmapsto \Gamma_f \in \Cor_k(X,Y)
\] is a faithful embedding.
 A \emph{presheaf with transfers} on $\Sch/k$ is a contravariant additive functor  $F: \Cor/k\to \Ab$. The category of all presheaves with transfers is denoted by
$\PrSh(\Cor/k)$. Composition with the graph functor yields a faithful embedding
\[
\PrSh(\Cor/k) \lang \PrSh(\Sch/k)
\]
from the category of presheaves with transfers to the category of presheaves on $\Sch/k$. We will consider presheaves with transfers as presheaves with the additional structure of transfer maps for finite correspondences.

\bigskip\noindent
If $X/k$ is smooth and connected, then $\Cor_k(X,Y)$ is the free abelian group generated by integral subschemes of $X\times_kY$ whose projection to $X$ is finite and surjective, see \cite[Proposition 3.3.5]{relcyc}.  For any $Y \in \Sch/k$ and any abelian group $A$ the \emph{Suslin homology}  $H_\bullet^S(Y,A)$  of $Y$ with values in $A$ (defined in \cite{SV96}) is the homotopy of the simplicial abelian group
\[
\begin{tikzcd}
\ds \cdots\arrow[r]\arrow[r,shift left=2]\arrow[r,shift right=2]&\Cor_k(\Delta^2_k,Y)\otimes A\arrow[r, shift left]\arrow[r, shift right]& \Cor_k(\Delta^1_k,Y)\otimes A\arrow[r]& \Cor_k(\Delta^0_k,Y)\otimes A,
\end{tikzcd}
\]
i.e., the homology of the complex
\[
 \cdots \Cor_k(\Delta_k^2,Y)\otimes  A \to \Cor_k(\Delta_k^1,Y)\otimes  A\to \Cor_k(\Delta_k^0,Y)\otimes  A  \to 0,
\]
where the differentials are induced by the alternating sums of the face maps. Likewise the \emph{Suslin cohomology} $H^\bullet_S(Y,A)$ of $Y$ with values in an abelian group $A$ is the cohomology of the complex
\[
\Hom_\Z(\Cor_k(\Delta_k^\bullet,Y), A).
\]

\bigskip\noindent
We consider the \emph{big tame site} $(\Sch/k)_t$ which consists of  the category $\Sch/k$  with tame coverings.  For $X\in \Sch/k$ and $F\in \Sh_t(\Sch/k)$, the cohomology of $X$ with values in $F$ coincides with the cohomology of the restriction of $F$ to the (small) tame site $(X/k)_t$ (cf.\ \cite[III,\,3.1]{Milne80}). A \emph{tame sheaf with transfers} is a presheaf with transfers which is a tame sheaf on $\Sch/k$ when considered as a presheaf on $\Sch/k$ via the graph functor. We denote the category of these by $\Sh_t(\Cor/k)$. For an abelian group $A$, the constant Zariski-sheaf $\underline{A}$ on $\Sch/k$ is already an \'{e}tale sheaf, in particular a tame sheaf. For $X\in \Sch/k$ we have
\[
\underline{A}(X)= \Cor_k(X,\Spec k)\otimes A,
\]
in particular, $\underline{A}$ has transfers in a natural way. Similar arguments as  for the Nisnevich and \'{e}tale cohomology show the following
\begin{lemma}\label{tamesheaftransfer}
The category $\Sh_t(\Cor/k)$ has sufficiently many injectives. An injective object $I\in \Sh_t(\Cor/k)$ is flabby as a sheaf in $\Sh_t(\Sch/k)$. In particular, we can calculate the tame cohomology of a tame sheaf with transfers by using an injective resolution in $\Sh_t(\Cor/k)$. For any $F\in \Sh_t(\Cor/k)$ the tame cohomology presheaves
\[
U \longmapsto H^n_t(U, F)
\]
are presheaves with transfers in a natural way.
\end{lemma}

\begin{proof}
As the tame site sits between the \'{e}tale and the Nisnevich site, the arguments of \cite[\S 3.1]{tria} (given there for sheaves with transfers on the category of smooth $k$-schemes) apply without change. See \cite[\S 10.3]{CiDe12} for more details and a treatment in the setting of $\Sch/k$.
\end{proof}

\ms\noi
 We make the following general observation. Let $(M^{\bullet\bullet},d',d'')$ be a commutative (or anti-commutative) double complex of abelian groups. For $(i,j)\in \Z \times \Z$ we put
\[
Z^{ij}= \ker(d': M^{ij}\to M^ {i+1,j}) \cap \ker(d'': M^{ij}\to M^ {i,j+1})
\]
\[
B^{ij}= \im(d'\circ d'': M^{i-1,j-1}\to M^{ij})\quad \text{and} \quad H^{ij}=Z^{ij}/B^{ij}.
\]
If all lines (i.e., the single complexes $M^{\bullet k}$ for all $k$) are exact, we obtain for all $(i,j)$ a natural homomorphism
\[
h_{i,j}: H^{i,j} \lang H^{i-1,j+1}
\]
as follows: Let $x\in H^{i,j}$ and let $m\in Z^{i,j}$ be a representing element. We choose $m'\in M^{i-1,j}$ with $d'(m')=m$ and then we define $h_{i,j}(x)\in H^{i-1,j+1}$ as the class of $d''(m')$. Our assumptions guarantee that $h_{i,j}$ is well-defined.

\ms\noindent
Fix an injective resolution $\Z/m\Z\to I^\bullet$ in the category of tame sheaves of $\Z/m\Z$-modules with transfers on $\Sch/k$. The finite correspondences
\[
\partial=\sum_{i=0}^n (-1)^i \partial_i \, :\  \Delta^{n-1} \to \Delta^n
\]
induce maps $\partial^*: I^\bullet(\Delta^n)\to I^\bullet (\Delta^{n-1})$. Consider  the double complex given by
\[
M^{ij}= \left\{
\begin{array}{ll}
I^i(\Delta^{-j}),&i\ge 0,\ j\le 0\\
\Z/m\Z,&i=-1,\ j \le 0\\
0,&\text{else}
\end{array}\right.
\]
with differentials as indicated below
\[
\begin{tikzcd}
\Z/m\Z\rar&I^0(\Delta^0)\rar{d}&I^1(\Delta^0)\rar{d}&I^2(\Delta^0)\\
\Z/m\Z\uar{0}\rar&I^0(\Delta^1)\uar{\partial^*}\rar{d}&I^1(\Delta^1)\uar{\partial^*}\rar{d}&I^2(\Delta^1)\uar{\partial^*}\\
\Z/m\Z\uar{\id}\rar&I^0(\Delta^2)\uar{\partial^*}\rar{d}&I^1(\Delta^2)\uar{\partial^*}\rar{d}&I^2(\Delta^2)\uar{\partial^*}\\
\Z/m\Z\uar{0}\rar&I^0(\Delta^3)\uar{\partial^*}\rar{d}&I^1(\Delta^3)\uar{\partial^*}\rar{d}&I^2(\Delta^3)\uar{\partial^*}.
\end{tikzcd} \eqno (*)
\]
Since $\Delta^n\cong \A^n_k$ and by homotopy invariance (\cref{homotopyinvariance}), we have
\[
H^i_t(\Delta^n,\Z/m\Z)\cong H^i_t(\Spec k,\Z/m\Z)= \left\{
\begin{array}{ll}
\Z/m\Z,&i=0\\
0,&i\ge 1.
\end{array}\right.
\]
Hence the lines of $(*)$ are exact and  we obtain a  map
\[
h_n: H^{n,-n} \to H^{n-1,-n+1} \to \ldots \to H^{0,0}= \Z/m\Z.
\]

\ms\noi
Returning to our original task,  let $\alpha \in H_n^S(X,\Z/m\Z)$ and $\beta \in H^n_t(X,\Z/m\Z)$ be given.
Let $a\in \Cor(\Delta^n,X)$ represent $\alpha$ and let $b\in I^n(X)$ represent $\beta$. By definition we have $db=0$ and $a\circ \partial = m\chi$ for some $\chi\in \Cor(\Delta^{n+1}, X)$.
We consider the pull back
\[
a^*(b) \in I^n(\Delta^n)= M^{n,-n}.
\]
We have
\[
d(a^*(b))=a^*(db)=a^*(0)=0, \ \text{and}
\]
\[
\partial^*(a^*(b))=(a\circ \partial)^*(b)=(m\chi)^*(b)=\chi^*(mb)=\chi^*(0)=0.
\]
Hence $a^*(b)\in Z^{n,-n}$, and we denote its class in $H^{n,-n}$ by $\overline{a^*(b)}$

\begin{definition}
We define
\[
\langle \alpha, \beta \rangle = h_n(\overline{a^*(b)}) \in \Z/m\Z.
\]
\end{definition}

\begin{lemma}
$\langle \alpha, \beta \rangle$ is independent of the choices made, hence we obtain a bilinear pairing
\[
H_n^S(X,\Z/m\Z) \times H^n_t(X,\Z/m\Z) \lang \Z/m\Z.
\]
\end{lemma}
\begin{proof}
Let $b'\in I^n(X)$ be another representative of $\beta$. If $n=0$, then $b=b'$ and there is nothing to prove. For $n\ge 1$, $b-b'=dc$, $c\in I^{n-1}(X)$, and we have to show that
\[
h_n(\overline{a^*(dc)})=0\ \text{for any}\ c\in I^{n-1}(X).
\]
By definition of the map $h_{n,-n}: H^{n,-n} \to H^{n-1,-n+1}$ above,
we have
\begin{multline*}
 h_{n,-n}(\overline{a^*(dc)})= h_{n,-n}(\overline{d(a^*(c))})= \overline{\partial^* (a^*(c))}=\overline{(a\circ \partial)^*(c)}=  \\
\overline{ (m\chi)^*(c)}=\overline{\chi^*(mc)}=\overline{\chi^*(0)}=0.
\end{multline*}
Since $h_n$ factors through $h_{n,-n}$, this shows the independence from the choice of~$b$.

\ms\noi
If $a'\in \Cor(\Delta^n,X)$ is another representative of $\alpha$, then there exist $x\in \Cor(\Delta^{n+1},X)$ and $y\in \Cor(\Delta^n,X)$ with
\[
a'-a = x \circ \partial + m y.
\]
It therefore remains to show that, for any $x\in \Cor(\Delta^{n+1},X)$, the map $(x\circ \partial)^*: H^n_t(X,\Z/m\Z)\to H^{n,-n}$ is zero. We have
\[
(x\circ \partial)^*(b)= \partial^*(x^*(b)).
\]
For $n=0$, $x^*(b)$ lies in $\ker(I^0(\Delta^1)\to I^1(\Delta^1))=H^0_t(\Delta^1,\Z/m\Z)=\Z/m\Z$ and
\[
\partial^*: H^0_t(\Delta^1,\Z/m\Z)=\Z/m\Z \lang H^{0,0}=\Z/m\Z
\]
is the zero map because $\partial^*=\partial_0^*-\partial_1^*=\id_{\Z/m\Z} -\id_{\Z/m\Z}=0$.

For $n\geq 1$, we have $H^n_t(\Delta^{n+1},\Z/m\Z)=0$, hence $x^*(b)=dc$ for some $c\in I^{n-1}(\Delta^{n+1})$ and we obtain
\[
\partial^*(x^*(b))=\partial^*(dc)\in B^{n,-n}
\]
showing the result.
\end{proof}

Using the pairing and noting that $\Z/m\Z$ is an injective $\Z/m\Z$-module, we obtain homomorphisms for all $n\ge 0$:
\[
\beta_n: H^n_t(X,\Z/m\Z)\lang \Hom(H_n^S(X,\Z/m\Z),\Z/m\Z)= H^n_S(X,\Z/m\Z).
\]

\begin{conjecture}\label{conj:suslinhom}
For any $m\ge 1$, the homomorphism $$\beta_n: H^n_t(X,\Z/m\Z)\lang  H^n_S(X,\Z/m\Z)$$ is an isomorphism of finite abelian groups for all $n\ge 0$.
\end{conjecture}

Up to comparing the maps, \cref{conj:suslinhom} is known if $(m,p)=1$ and all $n$ by \cite{SV96} and for general $m$ and $n=1$ by \cite{GS16}. We will address \cref{conj:suslinhom} in a forthcoming paper.
\bibliographystyle{./meinStil}
\bibliography{./citations}

\vskip1cm
\small

{\sc The Hebrew University of Jerusalem, Einstein Institute of Mathematics, Giv'at Ram, Jerusalem, 91904, Israel}

\textit{E-mail address:} {\tt katharin.hubner@mail.huji.ac.il}

\medskip
{\sc  Universit\"{a}t Heidelberg, Mathematisches Institut, Im Neuenheimer Feld 205, D-69120 Heidelberg, Deutschland}

\textit{E-mail address:} {\tt schmidt@mathi.uni-heidelberg.de}
\end{document}